\documentclass[a4paper]{article}
\usepackage[margin=30mm]{geometry}
\usepackage{amsmath}
\usepackage{amssymb}
\usepackage{amsthm}
\usepackage{titlesec}
\usepackage{titlefoot}
\titleformat*{\section}{\normalsize\bfseries}
\def\R{\mathbb{R}}
\def\e{{\varepsilon}}        

\def\p{\partial}
\newtheorem{thm}{Theorem}[section]
\newtheorem{lem}[thm]{Lemma}
\newtheorem{cor}[thm]{Corollary}
\newtheorem{prop}[thm]{Proposition}
\newtheorem{rem}[thm]{Remark}

\makeatletter
\@addtoreset{equation}{section}
 
  \makeatother
 
\begin{document}

\title{\normalsize{\bf HIGHER-ORDER ASYMPTOTIC PROFILES OF THE SOLUTIONS TO THE VISCOUS FORNBERG-WHITHAM EQUATION}}
\author{Ikki Fukuda and Kenta Itasaka}
\date{}
\maketitle

\footnote[0]{2020 Mathematics Subject Classification: 35B40, 35Q53.}

\begin{abstract}
We consider the initial value problem for the viscous Fornberg-Whitham equation which is one of the  nonlinear and nonlocal dispersive-dissipative equations. In this paper, we establish the global existence of the solutions and study its asymptotic behavior. We show that the solution to this problem converges to the self-similar solution to the Burgers equation called the nonlinear diffusion wave, due to the dissipation effect by the viscosity term. Moreover, we analyze the optimal asymptotic rate to the nonlinear diffusion wave and the detailed structure of the solution by constructing higher-order asymptotic profiles. 
Also, we investigate how the nonlocal dispersion term affects the asymptotic behavior of the solutions and compare the results with the ones of the KdV-Burgers equation.
\end{abstract}

\medskip
\noindent
{\bf Keywords:} Viscous Fornberg-Whitham equation, global existence, asymptotic behavior, \\
second asymptotic profile, optimal decay estimates, higher-order asymptotic profiles.

\section{Introduction}

We consider the initial value problem for the following viscous Fornberg-Whitham equation:
\begin{align}
\label{VFW}
\begin{split}
u_{t}+\beta uu_{x}+\int_{\R}Be^{-b|x-\xi|}u_{\xi}(\xi, t)d\xi&=\mu u_{xx}, \ \ x\in \R, \ t>0,\\
u(x, 0)&=u_{0}(x) , \ \ x\in \R, 
\end{split}
\end{align}
where $\beta\neq 0$ and $B, b, \mu>0$. The subscripts $t$ and $x$ denote the partial derivatives with respect to $t$ and $x$, respectively. If we take $\mu=0$ in \eqref{VFW}, we obtain the Fornberg-Whitham equation: 
\begin{align}
\label{FW}
\begin{split}
u_{t}+\beta uu_{x}+\int_{\R}Be^{-b|x-\xi|}u_{\xi}(\xi, t)d\xi&=0, \ \ x\in \R, \ t>0,\\
u(x, 0)&=u_{0}(x) , \ \ x\in \R. 
\end{split}
\end{align}
Fornberg-Whitham equation~\eqref{FW} was derived by Whitham~\cite{W67} and by Whitham and Fornberg~\cite{FW78} in the late 1900s, as a model for so-called ``breaking waves". Wave-breaking phenomena for equations with the nonlocal dispersion term (called the Whitham equation) was first studied by Seliger~\cite{S68}. He presented a formal argument that wave-breaking is possible for the Whitham equation. Also, the wave-breaking phenomena for solutions of the Whitham equation with a regular kernel like \eqref{FW} was studied in \cite{W74}. For another perspective on the Whitham equation, see~\cite{G78}. In addition, for the related results about more general nonlocal equations, see also \cite{NS94}. Moreover, Constantin and Escher~\cite{CE98} proved more mathematically rigorous condition for the blow-up of solutions to the Whitham equation. In~\cite{MLQ16}, their result was improved by Ma, Liu and Qu. Recently, Haziot~\cite{H17} obtained a different blow-up condition for~\eqref{FW} with $\beta=1$, which includes only the parameter $B$. Furthermore in~\cite{IP}, the second author proposed a new blow-up condition for~\eqref{FW} with $\beta=1$, which includes the parameters $B$ and $b$. Also, he investigated some relations between the Fornberg-Whitham equation~\eqref{FW} and the inviscid Burgers equation.

As related works, Tanaka \cite{T13} and H\"{o}rmann and Okamoto \cite{HO19} studied \eqref{FW} numerically. 
These results suggest that \eqref{FW} has blow-up solutions and global solutions depending on initial data and parameters $B$ and $b$. 
As we mentioned in the above paragraph, blow-up conditions for \eqref{FW} has been studied by many researchers. 
On the other hand, we have not had any mathematical result of the global existence for \eqref{FW} yet. 
As is well known, the solution of the KdV equation always exists globally in time (cf.~\cite{CKSTT03, G09, KPV96, Ki09}). This is because the nonlinear effect and the dispersive effect balance each other, and then the energy is conserved. From this point of view, to show the global existence of solutions to \eqref{FW}, it would be effective to investigate some relationship between the nonlinear effect and the dispersive effect in \eqref{FW} and compare the dispersion term with other type ones. 

On the other hand, the equation~\eqref{VFW} is the Fornberg-Whitham equation with the viscosity term $\mu u_{xx}$ representing the dissipation effect. In this case, by virtue of the dissipation effect, we can expect that the waves do not break down and the solutions to \eqref{VFW} exist globally in time. From these perspectives, in this paper, we would like to study the initial value problem  \eqref{VFW} and establish the global existence of the solutions with small initial data. In addition, we derive the asymptotic profile of large time behavior of the solution and the optimal convergence rate to its asymptotics. Especially, in order to investigate the structure of the solution in detail, we construct higher-order asymptotic profiles of the solution. Then, we compare the results with the ones of the KdV-Burgers equation.

Before we state our main results, let us refer to some known results about the large time asymptotic behavior for solutions to the related problems. First, to analyze \eqref{VFW}, we transform the nonlocal dispersion term by using the expression 
\begin{equation*}
\int_{\R}Be^{-b|x-\xi|}u_{\xi}(\xi, t)d\xi=\mathcal{F}^{-1}\left[\frac{i2Bb\xi}{b^{2}+\xi^{2}}\hat{u}(\xi)\right](x)=2Bb(b^{2}-\p_{x}^{2})^{-1}u_{x}.
\end{equation*}
Then, we can rewrite \eqref{VFW} as follows: 
\begin{equation}\label{VFW'}
u_{t}+\beta uu_{x}+2Bb(b^{2}-\p_{x}^{2})^{-1}u_{x}=\mu u_{xx}. 
\end{equation}
From \eqref{VFW'}, we expect that \eqref{VFW} has a similar structure to the following Burgers equation: 
\begin{align}
\label{B}
\begin{split}
u_{t}+\alpha u_{x}+\beta uu_{x}&=\mu u_{xx}, \ \ x\in \R, \ t>0,\\
u(x, 0)&=u_{0}(x) , \ \ x\in \R, 
\end{split}
\end{align}
where $\alpha \in \R$, while $\beta$ and $\mu$ are defined in \eqref{VFW}. When $\alpha=0$, the initial value problem for the Burgers equation \eqref{B} is well studied by many researchers (e.g.~\cite{C51, H50, K07, K87, L00, MN04, N86}). Moreover, the asymptotic behavior of the solution is already obtained. In particular, by using the change of variable, the result by~\cite{L00} can be modified for the problem~\eqref{B} with $\alpha \in \R$. Actually, the solution of \eqref{B} converges to the nonlinear diffusion wave which is a modification of the self-similar solution of the Burgers equation and is defined by
\begin{equation}
\label{chi1}
\chi(x, t):=\frac{1}{\sqrt{1+t}} \chi_{*} \biggl(\frac{x-\alpha(1+t)}{\sqrt{1+t}}\biggl), \ x\in \R,\ \ t>0,\ \ \alpha \in \R,
\end{equation}
where
\begin{equation}
\label{chi2}
\chi_{*}(x):=\frac{\sqrt{\mu}}{\beta}\frac{(e^{\frac{\beta M}{2\mu}}-1)e^{-\frac{x^{2}}{4\mu}}}{\sqrt{\pi}+(e^{\frac{\beta M}{2\mu}}-1)\int_{x/\sqrt{4\mu}}^{\infty}e^{-y^{2}}dy}, \ M:=\int_{\R}u_{0}(x)dx, \ \beta \neq0. 
\end{equation}
More precisely, if $u_{0} \in L_{1}^{1}(\R) \cap H^{1}(\R)$ and $\|u_{0}\|_{L_{1}^{1}}+\|u_{0}\|_{H^{1}}$ is sufficiently small, then the solution to \eqref{B} satisfies 
\begin{equation}
\label{asymp.B}
\|u(\cdot, t)-\chi(\cdot, t)\|_{L^{\infty}} \le C(1+t)^{-1}, \ \ t\ge0.
\end{equation}
Here, for $k\ge0$, we denote the weighted Lebesgue spaces $L_{k}^{1}(\R)$ as a subset of $L^{1}(\R)$ whose elements satisfy $\|u_{0}\|_{L_{k}^{1}}:=\int_{\R}|u_{0}(x)|(1+|x|)^{k}dx <\infty$. Also, by the Hopf-Cole transformation (cf.~\cite{C51, H50}), we can see that $\chi(x, t)$ satisfies the following Burgers equation and the following conservation law:
\begin{equation}
\label{BC}
\chi_{t}+\left(\alpha \chi+\frac{\beta}{2}\chi^{2}\right)_{x}=\mu \chi_{xx}, \ \ \int_{\R}\chi(x, t)dx=M.  
\end{equation}

Next, we consider the following equation:
\begin{align}
\label{KdVB}
\begin{split}
u_{t}+\alpha u_{x}+\beta uu_{x}+\gamma u_{xxx}&=\mu u_{xx}, \ \ x\in \R, \ t>0,\\
u(x, 0)&=u_{0}(x) , \ \ x\in \R, 
\end{split}
\end{align}
where $\alpha, \gamma \in \R$, while $\beta$ and $\mu$ are defined in \eqref{VFW}. This equation is called the KdV-Burgers equation, which can be regarded as the Burgers equation with dispersive perturbation and as the KdV equation with viscosity term.
There are many results about the asymptotic behavior of the solution to the KdV-Burgers equation with $\alpha=0$ (cf.~\cite{F19-1, F19-2, HN06, KP05, K97, K99}). 
In particular, by using the change of variable, we have the following estimate from the result of~\cite{F19-1}: 
If $u_{0}\in L^{1}_{1}(\R)\cap H^{3}(\R)$ and $\|u_{0}\|_{L^{1}_{1}}+\|u_{0}\|_{H^{3}}$ is sufficiently small, then the solution to \eqref{KdVB} satisfies 
\begin{equation}
\label{asymp.KdVB}
\|u(\cdot, t)-\chi(\cdot, t)\|_{L^{\infty}}=(\tilde{C}+o(1))(1+t)^{-1}\log(1+t) \ \ as \ \ t\to \infty, 
\end{equation}
where $\chi(x, t)$ is defined by~\eqref{chi1}, while $\tilde{C}$ is a certain positive constant depending on $\beta$, $\gamma$ and $M$ under $M\gamma \neq0$.
Therefore, we see that the solution $u(x, t)$ to \eqref{KdVB} tends to the nonlinear diffusion wave $\chi(x, t)$ at the optimal rate of $t^{-1}\log t$ in the $L^{\infty}$-sense if $M\gamma \neq0$. 
We note that the asymptotic rate to the nonlinear diffusion wave given in \eqref{asymp.KdVB} is slower than \eqref{asymp.B} due to the dispersion term $\gamma u_{xxx}$. 
On the other hand, our target equation \eqref{VFW} can be regarded as the Burgers equation with different type dispersion term $\int_{\R}Be^{-b|x-\xi|}u_{\xi}(\xi, t)d\xi$. For this reason, we are interested in how this dispersion term affects the asymptotic behavior of the solution to \eqref{VFW}.

\medskip
For $s\ge1$ and $k\ge0$, we set $E_{s, k}:=\|u_{0}\|_{H^{s}}+\|u_{0}\|_{L^{1}_{k}}$. Then, we obtain the following result: 
\begin{thm}
\label{thm.main1}
Let $s\ge1$. Assume that $u_{0}\in L^{1}(\R)\cap H^{s}(\R)$ and $E_{s, 0}$ is sufficiently small. Then \eqref{VFW} has a unique global solution $u(x, t) \in C^{0}([0, \infty); H^{s})$. Moreover, if $u_{0}\in L^{1}_{1}(\R)\cap H^{s}(\R)$, for all $\e>0$, the estimate 
\begin{equation}
\label{main1}
\|\p_{x}^{l}(u(\cdot, t)-\chi(\cdot, t))\|_{L^{p}}\le CE_{s, 1}(1+t)^{-1+\frac{1}{2p}-\frac{l}{2}+\e}, \ t\ge0
\end{equation}
holds for any $p\in [2, \infty]$ and integer $0\le l\le s-1$, where $\chi(x, t)$ is defined by \eqref{chi1} with $\displaystyle \alpha=\frac{2B}{b}$.
\end{thm}
\begin{rem}
{\rm A result analogous to the above theorem has also been obtained for more general dispersive-dissipative type nonlinear equations in the $L^{\infty}$-sense (for details, see Theorem~4.13 in~\cite{HKNS06}). On the other hand, our result \eqref{main1} gives us the more concrete asymptotic formula for the solutions to \eqref{VFW}, which includes the estimate for the derivative of the solutions in the $L^{p}$-sense.
}
\end{rem}

Furthermore, we can construct the second asymptotic profile of the solution to \eqref{VFW} which is the leading term of $u-\chi$, and derive the optimal asymptotic rate to the nonlinear diffusion wave under the additional regularity assumption on the initial data. To state such a result, we define the following function
\begin{equation}\label{second1}
V(x, t):= -\kappa dV_{*}\biggl(\frac{x-\alpha(1+t)}{\sqrt{1+t}}\biggl)(1+t)^{-1}\log(1+t), \ \ \alpha=\frac{2B}{b},
\end{equation}
where
\begin{align}\label{second2}
&V_{*}(x):= \frac{1}{\sqrt{4\pi \mu}}\frac{d}{dx}(\eta_{*}(x)e^{-\frac{x^{2}}{4\mu}}), \ \ \eta_{*}(x):= \exp \biggl(\frac{\beta}{2\mu}\int_{-\infty}^{x}\chi_{*}(y)dy\biggl), \\
&d:= \int_{\R}(\eta_{*}(y))^{-1}(\chi_{*}(y))^{3}dy, \ \ \kappa := \frac{\beta^{2}B}{4b^{3}\mu^{2}}=\frac{\alpha \beta^{2}}{8b^{2}\mu^{2}}.  \label{second3}
\end{align}
Then, our second main result of this paper is as follows:
\begin{thm}
\label{thm.main2}
Let $s\ge2$. Assume that $u_{0}\in L^{1}_{1}(\R)\cap H^{s}(\R)$ and $E_{s, 0}$ is sufficiently small. Then, for the solution to~\eqref{VFW}, the estimate
\begin{equation}
\label{main2}
\|\p_{x}^{l}(u(\cdot, t)-\chi(\cdot, t)-V(\cdot, t))\|_{L^{p}}\le CE_{s, 1}(1+t)^{-1+\frac{1}{2p}-\frac{l}{2}}, \ t\ge1
\end{equation}
holds for any $p\in [2, \infty]$ and integer $0\le l\le s-2$, where $\chi(x, t)$ is defined by~\eqref{chi1} with $\displaystyle \alpha=\frac{2B}{b}$, while $V(x, t)$ is defined by~\eqref{second1}.
\end{thm}
\noindent
In view of the second asymptotic profile, we are able to obtain the optimal asymptotic rate to the nonlinear diffusion wave $\chi(x, t)$ as follows: 
\begin{cor}\label{cor.main2}
Under the same assumptions in Theorem {\rm \ref{thm.main2}}, if $M\neq0$, the estimate 
\begin{equation}\label{main.cor}
\|\p_{x}^{l}(u(\cdot, t)-\chi(\cdot, t))\|_{L^{p}}=(C_{0}+o(1))(1+t)^{-1+\frac{1}{2p}-\frac{l}{2}}\log(1+t) \ \ as \ \ t\to \infty
\end{equation}
holds for $p\in [2, \infty]$ and integer $0\le l\le s-2$, where $C_{0}:=|\kappa d|\|\p_{x}^{l}V_{*}\|_{L^{p}}$ is a positive constant. 
\end{cor}
\begin{rem}
{\rm
From~\eqref{main1}, \eqref{main2} and~\eqref{main.cor}, we can see that the first and second asymptotic profiles of the solution to~\eqref{VFW} are similar to Burgers type equations such as generalized Burgers equation, generalized KdV-Burgers equation and BBM-Burgers equation, and also damped wave equation with a convection term (cf.~\cite{F19-1, HKN07, K07, KU17}). 
}
\end{rem}

We note that compared with the term $\gamma u_{xxx}$ in~\eqref{KdVB}, the term $\int_{\R}Be^{-b|x-\xi|}u_{\xi}(\xi, t)d\xi$ in \eqref{VFW} acts not only as a dispersive perturbation but also as a linear convection term. The main reason for it is that we can split the nonlocal dispersion term into dispersion and convection parts as $2Bb(b^{2}-\p_{x}^{2})^{-1}u_{x}=\alpha(b^{2}-\p_{x}^{2})^{-1}\p_{x}^{3}u+\alpha u_{x}$ (for details, we discuss it later). However, from Theorem~\ref{thm.main1} and Theorem~\ref{thm.main2}, the first and second asymptotic profiles of the solution to \eqref{VFW} are essentially the same as that of the KdV-Burgers equation \eqref{KdVB}, and the effect of $\alpha(b^{2}-\p_{x}^{2})^{-1}\p_{x}^{3}u$ is almost the same as $\gamma u_{xxx}$. To further investigate the effect of the nonlocal dispersion term, we would like to analyze the more detailed structure of the solution. Actually, we construct the third asymptotic profile of the solution and derive the optimal decay estimate for $u-\chi-V$. To state the next result, let us define the following new functions $W(x, t)$ and $\Psi(x, t)$. First, we define 
\begin{equation}\label{third1}
W(x, t):=\theta V_{*}\biggl(\frac{x-\alpha(1+t)}{\sqrt{1+t}}\biggl)(1+t)^{-1}, \ \ \alpha=\frac{2B}{b},
\end{equation} 
where 
\begin{align}
&\theta:=\int_{\R}z_{0}(x)dx+\int_{0}^{\infty}\int_{\R}\rho(x, t)dxdt, \ \ z_{0}(x):=\eta(x, 0)^{-1}\int_{-\infty}^{x}(u_{0}(y)-\chi(y, 0))dy, \label{initial}\\
&\rho(x, t):=-\eta(x, t)^{-1}\left(\frac{\beta}{2}(u-\chi)^{2}+\frac{2B}{b}(b^{2}-\p_{x}^{2})^{-1}\p_{x}^{2}(u-\chi)+\frac{2B}{b^{3}}(b^{2}-\p_{x}^{2})^{-1}\p_{x}^{4}\chi\right)(x, t),\label{rho}\\
&\eta(x, t):=\eta_{*}\left(\frac{x-\alpha(1+t)}{\sqrt{1+t}}\right)=\exp \biggl(\frac{\beta}{2\mu}\int_{-\infty}^{x}\chi(y, t)dy\biggl), \label{eta}
\end{align}
with $V_{*}(x)$ and $\eta_{*}(x)$ being defined by \eqref{second2}. Next, we define 
\begin{equation}\label{third2}
\Psi(x, t):=\Psi_{*}\left(\frac{x-\alpha(1+t)}{\sqrt{1+t}}\right)(1+t)^{-1}, \ \ \alpha=\frac{2B}{b},
\end{equation}
where
\begin{align}
&\Psi_{*}(x):=\frac{d}{dx}\left(\eta_{*}(x)\int_{0}^{1}(G(1-\tau)*F(\tau))(x)d\tau\right), \ \ 
G(x, t):=\frac{1}{\sqrt{4\pi \mu t}}e^{-\frac{x^{2}}{4\mu t}}, \label{Gauss} \\
&F(x, \tau):=F_{*}\left(\frac{x}{\sqrt{\tau}}\right)\tau^{-\frac{3}{2}}, \ F_{*}(x):=\frac{2B}{b^{3}}\eta_{*}(x)^{-1}\chi_{*}''(x)-\frac{\kappa d}{\sqrt{4\pi \mu}}e^{-\frac{x^{2}}{4\mu}}, \label{FF}
\end{align}
with $\kappa$ and $d$ being defined by \eqref{second3}. Finally, combining $W(x, t)$ and $\Psi(x, t)$, we set 
\begin{equation}\label{third} 
Q(x, t):=W(x, t)+\Psi(x, t). 
\end{equation}
Then, the third asymptotic profile of the solution to \eqref{VFW} is given by the above function $Q(x, t)$. Actually, we have the following asymptotic relation: 
\begin{thm}
\label{thm.main3}
Let $s\ge3$. Assume that $u_{0}\in L^{1}_{1}(\R)\cap H^{s}(\R)$, $z_{0}\in L^{1}_{1}(\R)$ and $E_{s, 0}$ is sufficiently small. Then, the solution to~\eqref{VFW} satisfies 
\begin{equation}
\label{main3}
\lim_{t\to \infty}(1+t)^{1-\frac{1}{2p}+\frac{l}{2}}\|\p_{x}^{l}(u(\cdot, t)-\chi(\cdot, t)-V(\cdot, t)-Q(\cdot, t))\|_{L^{p}}=0
\end{equation}
for any $p\in [2, \infty]$ and integer $0\le l\le s-3$, where $\chi(x, t)$ and $V(x, t)$ are defined by~\eqref{chi1} and~\eqref{second1}, respectively, while $Q(x, t)$ is defined by \eqref{third}. 
\end{thm}
\noindent
Similar to Corollary~\ref{cor.main2}, the following optimal decay estimate of $u-\chi-V$ can be obtained: 
\begin{cor}
\label{cor.main3}
Under the same assumptions in Theorem~\ref{thm.main3}, the estimate 
\begin{equation}
\label{main.cor2}
\|\p_{x}^{l}(u(\cdot, t)-\chi(\cdot, t)-V(\cdot, t))\|_{L^{p}}=(c_{0}+o(1))(1+t)^{-1+\frac{1}{2p}-\frac{l}{2}} \ \ as \ \ t\to \infty
\end{equation}
holds for any $p\in [2, \infty]$ and integer $0\le l\le s-3$, where $c_{0}:=\|\p_{x}^{l}(\theta V_{*}+\Psi_{*})\|_{L^{p}}$ is a constant. 
\end{cor}
\begin{rem}
{\rm
Our method of construction for the third asymptotic profile of the solution can also be applied to the KdV-Burgers equation \eqref{KdVB}. In particular, take the same initial data $u_{0}(x)$ and consider the case $\displaystyle \alpha=\frac{2B}{b}$, same $\beta$ and $\displaystyle \gamma=\frac{2B}{b^{3}}$ in \eqref{KdVB}, that is 
\begin{align}\tag{\ref{KdVB}}
\begin{split}
\tilde{u}_{t}+\frac{2B}{b}\tilde{u}_{x}+\beta \tilde{u}\tilde{u}_{x}+\frac{2B}{b^{3}}\tilde{u}_{xxx}&=\mu \tilde{u}_{xx}, \ \ x\in \R, \ t>0,\\
\tilde{u}(x, 0)&=u_{0}(x) , \ \ x\in \R. 
\end{split}
\end{align}
In this case, under the same assumptions in Theorem~\ref{thm.main3}, we can show the following formula:  
\begin{equation}\label{KdVB-AP}
\lim_{t\to \infty}(1+t)^{1-\frac{1}{2p}+\frac{l}{2}}\|\p_{x}^{l}(\tilde{u}(\cdot, t)-\chi(\cdot, t)-V(\cdot, t)-\tilde{Q}(\cdot, t))\|_{L^{p}}=0, 
\end{equation}
where the first asymptotic profile $\chi(x, t)$ and the second asymptotic profile $V(x, t)$ are exactly the same functions as in the case of \eqref{VFW}. On the other hand, $\tilde{Q}(x, t)$ is defined as 
\[\tilde{Q}(x, t):=\tilde{W}(x, t)+\Psi(x, t)\]
with $\Psi(x, t)$ being defined by \eqref{third2} and $\tilde{W}(x, t)$ is defined by 
\begin{align}\label{tilde}
\begin{split}
&\tilde{W}(x, t):=\tilde{\theta} V_{*}\biggl(\frac{x-\alpha(1+t)}{\sqrt{1+t}}\biggl)(1+t)^{-1}, \ \ \tilde{\theta}:=\int_{\R}z_{0}(x)dx+\int_{0}^{\infty}\int_{\R}\tilde{\rho}(x, t)dxdt, \\ 
&\tilde{\rho}(x, t):=-\eta(x, t)^{-1}\left(\frac{\beta}{2}(\tilde{u}-\chi)^{2}+\frac{2B}{b^{3}}\p_{x}^{2}(\tilde{u}-\chi)\right)(x, t)
\end{split}
\end{align}
with $V_{*}(x)$, $z_{0}(x)$ and $\eta(x, t)$ being defined by \eqref{second2}, \eqref{initial} and \eqref{eta}, respectively. As can be seen from the definition of $W (x, t)$ (defined by \eqref{third1}) and $\tilde{W}(x, t)$, the effect of the dispersion terms appear on the amplitude $\theta$ (defined by \eqref{initial}) and $\tilde{\theta}$ in the third asymptotic profile, respectively. Here, we note that $\theta$ and $\tilde{\theta}$ are not always equal. This result implies that even if the parameters are selected so that the first asymptotic profile and the second asymptotic profile of \eqref{KdVB} match those of \eqref{VFW}, the third asymptotic profile $\tilde{Q}(x, t)$ does not always equal to $Q(x, t)$. 
}
\end{rem}

\begin{rem}
{\rm
We consider the effect of the nonlocal dispersion term in \eqref{VFW} for more higher-order asymptotic profiles.
Comparing the case of \eqref{VFW} and the case of \eqref{KdVB}, from \eqref{main3} and \eqref{KdVB-AP}, we can see that the nonlocal term $(b^{2}-\p_{x}^{2})^{-1}\p_{x}^{2}(u-\chi)$ and $(b^{2}-\p_{x}^{2})^{-1}\p_{x}^{4}\chi$ in $\rho(x,t)$ (defined by \eqref{rho}) appeared from the dispersion term are replaced by $\p_{x}^{2}(\tilde{u}-\chi)$ in $\tilde{\rho}(x, t)$ (defined by \eqref{tilde}). In particular, for \eqref{KdVB}, the term corresponding to higher-order derivative such as $\p_{x}^{4}\chi$ does not appear in $\tilde{\rho}(x, t)$. On the other hand for \eqref{VFW}, higher-order derivative term $(b^{2}-\p_{x}^{2})^{-1}\p_{x}^{4}\chi$ appears. 
This is because the nonlocal dispersion term has the following expression:
\[\int_{\R}Be^{-b|x-\xi|}u_{\xi}(\xi, t)d\xi=2Bb(b^{2}-\p_{x}^{2})^{-1}u_{x}=\frac{2B}{b}\p_{x}u + \frac{2B}{b^{3}}\p_{x}^{3}u + \frac{2B}{b^{3}}(b^{2}-\p_{x}^{2})^{-1}\p_{x}^{5}u.\]
For Burgers type equations, the time decay of the solution gets faster with spatial derivative.
Also, thinking about higher-order asymptotic profiles corresponds to viewing solutions in terms of faster time decay.
Hence, although the effect of the fifth-order derivative term do not appear in the first asymptotic profile and the second asymptotic profile,
the effect and the difference between \eqref{VFW} and \eqref{KdVB} appear in the third asymptotic profile.
More generally, we note that the following formal series expansion: 
\[\frac{i2Bb\xi}{b^{2}+\xi^{2}}=\frac{2B}{b}\left\{i\xi-\frac{i\xi^{3}}{b^{2}}+\frac{i\xi^{5}}{b^{4}}-\frac{i\xi^{7}}{b^{6}}+\cdots\right\}.\]
Thus, the nonlocal dispersion term can be considered  formally as follows: 
\[\int_{\R}Be^{-b|x-\xi|}u_{\xi}(\xi, t)d\xi=2Bb(b^{2}-\p_{x}^{2})^{-1}u_{x}=\frac{2B}{b}\left\{\p_{x}u+\frac{1}{b^{2}}\p_{x}^{3}u+\frac{1}{b^{4}}\p_{x}^{5}u+\frac{1}{b^{6}}\p_{x}^{7}u+\cdots\right\}.\]
Therefore, more higher-order asymptotic profiles may be affected by more higher-order derivatives.
Also, the difference between the effect of the nonlocal dispersion term and the effect of the ordinary dispersion term $u_{xxx}$ would appear from the third, fourth or more higher-order asymptotic profiles.
}
\end{rem}

This paper is organized as follows. In section 2, we prove the global existence and decay estimates for the solutions to~\eqref{VFW}. In Section 3, we introduce some decay estimates for asymptotic functions and basic properties for an auxiliary problem. In Section 4, we derive the asymptotic behavior of the solutions to~\eqref{VFW}, i.e., we prove Theorem~\ref{thm.main1}. Next, we derive the second asymptotic profile $V(x, t)$ and give the proof of Theorem~\ref{thm.main2} in Section 5. Finally, we show that the third asymptotic profile is given by $Q(x, t)$ in the last section, i.e., we prove Theorem~\ref{thm.main3} in Section 6. The main difficulty of the proofs of Theorem~\ref{thm.main1}, Theorem~\ref{thm.main2} and Theorem~\ref{thm.main3} is how to treat the nonlocal dispersion term $\int_{\R}Be^{-b|x-\xi|}u_{\xi}(\xi, t)d\xi$. To avoid that difficulty, based on the above expansion, we transform this term to $\alpha(b^{2}-\p_{x}^{2})^{-1}\p_{x}^{3}u+\alpha u_{x}$ and apply the idea of the asymptotic analysis for the KdV-Burgers equation used in~\cite{F19-1, KP05}. 

\medskip
\par\noindent
\textbf{\bf{Notations.}} In this paper, for $1\le p \le \infty$, $L^{p}(\R)$ denotes the usual Lebesgue spaces. Moreover, for $k\ge0$, we define the following weighted Lebesgue spaces: 
\begin{equation*}
L^{1}_{k}(\R):=\biggl\{ f \in L^{1}(\R); \ \|f\|_{L^{1}_{k}}:=\int_{\R}|f(x)|(1+|x|)^{k}dx<\infty \biggl\}.
\end{equation*}

In the following, for $f, g \in L^{2}(\R)\cap L^{1}(\R)$, we denote the Fourier transform of $f$ and the inverse Fourier transform of $g$ as follows:
\begin{align*}
\hat{f}(\xi):=\mathcal{F}[f](\xi)=\frac{1}{\sqrt{2\pi}}\int_{\R}e^{-ix\xi}f(x)dx, \ \ \check{g}(x):=\mathcal{F}^{-1}[g](x)=\frac{1}{\sqrt{2\pi}}\int_{\R}e^{ix\xi}g(\xi)d\xi.
\end{align*}
Then, for $s\ge0$, we define the Sobolev spaces by 
\begin{equation*}
H^{s}(\R):=\biggl\{f\in L^{2}(\R); \ \|f\|_{H^{s}}:=\biggl(\int_{\R}(1+|\xi|^{2})^{s}|\hat{f}(\xi)|^{2}d\xi\biggl)^{1/2}<\infty \biggl\}. 
\end{equation*}
To express Sobolev spaces, for $1\le p\le \infty$, we also set 
\begin{equation*}
W^{m,p}(\R):=\biggl\{ f\in L^{p}(\R); \ \|f\|_{W^{m,p}}:=\biggl(\sum_{n=0}^{m}\| \p_{x}^{n}f\|_{L^{p}}^{p}\biggl)^{1/p}<\infty \biggl\}.
\end{equation*}

Throughout this paper, $C$ denotes various positive constants, which may vary from line to line during computations. Also, it may depend on the norm of the initial data or other parameters. However, we note that it does not depend on the space variable $x$ and the time variable $t$.

\section{Global Existence and Decay Estimates}  

In this section, we shall prove the global existence and decay estimates for the solutions to~\eqref{VFW}. 
First, we introduce the Green function associated with the linear part of the equation in~\eqref{VFW}:
\begin{equation}\label{2-1}
T(x, t):=\mathcal{F}^{-1}[e^{-\mu t|\xi|^{2}-\frac{i2Bbt\xi}{b^{2}+\xi^{2}}}](x).
\end{equation}
We can show the following estimate of $T(x, t)$. The proof is completely the same as Lemma 2.2 in \cite{F19-1}.
\begin{lem}
\label{lem2-1}
Let $s\ge1$. Suppose $f\in L^{1}(\R) \cap H^{s}(\R)$. Then the estimate
\begin{equation}
\label{2-2}
\| \p^{l}_{x}(T(t)*f)\|_{L^{2}} \le C(1+t)^{-\frac{1}{4}-\frac{l}{2}}\|f\|_{L^{1}}+Ce^{-\mu t}\| \p^{l}_{x}f\|_{L^{2}}, \ \ t\ge0
\end{equation}
holds for any integer $0\le l \le s$.
\end{lem}

\noindent
Now, let us prove the global existence and the decay estimates of the solutions to \eqref{VFW}. The proof of the following theorem is almost the same as Proposition~2.3 in \cite{F19-2}. 
\begin{thm}
\label{GE}
Let $s\ge1$. Assume that $u_{0}\in L^{1}(\R)\cap H^{s}(\R)$ and $E_{s, 0}$ is sufficiently small. Then \eqref{VFW} has a unique global solution $u(x, t) \in C^{0}([0, \infty); H^{s})$. Moreover, the solution satisfies
\begin{equation}
\label{2-5}
\| \p^{l}_{x}u(\cdot, t)\|_{L^{2}}\le CE_{s, 0}(1+t)^{-\frac{1}{4}-\frac{l}{2}}, \ \ t\ge0 
\end{equation}
for any integer $0\le l \le s$. In particular, we get 
\begin{equation}
\label{2-6}
\| \p^{l}_{x}u(\cdot, t)\|_{L^{\infty}}\le CE_{s, 0}(1+t)^{-\frac{1}{2}-\frac{l}{2}}, \ \ t\ge0
\end{equation}
for any integer $0\le l \le s-1$. 
\end{thm}
\begin{proof}
We consider the following integral equation associated with the initial value problem \eqref{VFW}:
\begin{align}
\label{2-7}
\begin{split}
u(t)&=T(t)*u_{0}-\frac{\beta}{2}\int_{0}^{t}T(t-\tau)*((u^{2})_{x})(\tau)d\tau \\
&=T(t)*u_{0}-\frac{\beta}{2}\int_{0}^{t}(\p_{x}T(t-\tau))*(u^{2})(\tau)d\tau. 
\end{split}
\end{align}
We solve this integral equation by using the contraction mapping principle for the mapping 
\begin{equation}
\label{2-8}
N[u]:=T(t)*u_{0}-\frac{\beta}{2}\int_{0}^{t}(\p_{x}T(t-\tau))*(u^{2})(\tau)d\tau.
\end{equation}
We set $N_{0}:=T(t)*u_{0}$. Let us introduce the Banach space $X$ as follows:
\begin{equation}
\label{2-9}
X:=\biggl\{u\in C^{0}([0, \infty); H^{s}); \ \|u\|_{X}:=\sum_{l=0}^{s}\sup_{t\ge0}(1+t)^{\frac{1}{4}+\frac{l}{2}}\|\p_{x}^{l}u(\cdot, t)\|_{L^{2}}<\infty \biggl\}.
\end{equation}
From Lemma \ref{lem2-1}, we have
\begin{equation}
\label{2-10}
\exists C_{0}>0 \ \ s.t. \ \ \|N_{0}\|_{X}\le C_{0}E_{s, 0}.
\end{equation}
We apply the contraction mapping principle to \eqref{2-8} on the closed subset $Y$ of $X$ below: 
\begin{equation*}
Y:=\{u\in X; \ \|u\|_{X}\le2C_{0}E_{s, 0}\}.
\end{equation*}
Then it is sufficient to show the following estimates: 
\begin{equation}
\label{2-11}
\|N[u]\|_{X}\le2C_{0}E_{s, 0}, 
\end{equation}
\begin{equation}
\label{2-12}
\|N[u]-N[v]\|_{X}\le \frac{1}{2}\|u-v\|_{X} 
\end{equation}
for $u, v \in Y$. If we have shown \eqref{2-11} and \eqref{2-12}, by using the contraction mapping principle, we see that \eqref{2-7} has a unique global solution in $Y$. 

Here and later $E_{s, 0}$ is assumed to be small. First, from the Sobolev inequality 
\begin{equation*}
\|f\|_{L^{\infty}}\le \sqrt{2}\|f\|_{L^{2}}^{1/2}\|f'\|_{L^{2}}^{1/2}, \ \ f\in H^{1}(\R)
\end{equation*}
for $0\le l \le s-1$, we have 
\begin{equation}
\label{2-13}
\|\p_{x}^{l}u(\cdot, t)\|_{L^{\infty}}\le \|u\|_{X}(1+t)^{-\frac{1}{2}-\frac{l}{2}}.
\end{equation}
Before proving \eqref{2-11} and \eqref{2-12}, we prepare the following estimates for $0\le l \le s$, $u, v \in Y$:
\begin{align}
\label{2-14}
\|\p_{x}^{l}(u^{2}-v^{2})(\cdot, t)\|_{L^{1}}\le&C(\|u\|_{X}+\|v\|_{X})\|u-v\|_{X}(1+t)^{-\frac{1}{2}-\frac{l}{2}}, \\
\label{2-15}
\|\p_{x}^{l}(u^{2}-v^{2})(\cdot, t)\|_{L^{2}}\le&C(\|u\|_{X}+\|v\|_{X})\|u-v\|_{X}(1+t)^{-\frac{3}{4}-\frac{l}{2}}. 
\end{align}
We shall prove only \eqref{2-14}, since we can prove \eqref{2-15} in the same way. We have from \eqref{2-9} and \eqref{2-13}
\begin{align*}
\begin{split}
\|\p_{x}^{l}(u^{2}-v^{2})(\cdot, t)\|_{L^{1}}&= \|\p_{x}^{l}((u+v)(u-v))(\cdot, t)\|_{L^{1}}\\
&\le C\sum_{m=0}^{l}(\|\p_{x}^{l-m}u(\cdot, t)\|_{L^{2}}+\|\p_{x}^{l-m}v(\cdot, t)\|_{L^{2}})\|\p_{x}^{m}(u-v)(\cdot, t)\|_{L^{2}} \\
&\le C\sum_{m=0}^{l}(\|u\|_{X}+\|v\|_{X})(1+t)^{-\frac{1}{4}-\frac{l-m}{2}}\|u-v\|_{X}(1+t)^{-\frac{1}{4}-\frac{m}{2}}\\
&\le C(\|u\|_{X}+\|v\|_{X})\|u-v\|_{X}(1+t)^{-\frac{1}{2}-\frac{l}{2}}.
\end{split}
\end{align*}

Now we prove \eqref{2-11} and \eqref{2-12}. Using \eqref{2-8}, we obtain 
\begin{equation}
\label{2-16}
(N[u]-N[v])(t)=-\int_{0}^{t}(\p_{x}T(t-\tau))*(u^{2}-v^{2})(\tau)d\tau=:I(x, t). 
\end{equation}
By Plancherel's theorem, we have
\begin{align}
\label{2-17}
\begin{split}
\|\p_{x}^{l}I(\cdot, t)\|_{L^{2}} &\le \|(i\xi)^{l}\hat{I}(\xi, t)\|_{L^{2}(|\xi|\le1)}+\|(i\xi)^{l}\hat{I}(\xi, t)\|_{L^{2}(|\xi|\ge1)}=:I_{1}+I_{2}. 
\end{split}
\end{align}
Since 
\begin{equation*}
\int_{|\xi|\le1}|\xi|^{j}e^{-2(t-\tau)|\xi|^{2}}d\xi \le C(1+t-\tau)^{-\frac{j}{2}-\frac{1}{2}}, \ \ j\ge0, 
\end{equation*}
and \eqref{2-14}, we have
\begin{align}
\label{2-18}
\begin{split}
I_{1}&\le C\int_{0}^{t}\|(i\xi)^{l+1}e^{-\mu(t-\tau)|\xi|^{2}-\frac{i2Bb(t-\tau)\xi}{b^{2}+\xi^{2}}}\mathcal{F}[u^{2}-v^{2}](\xi, \tau)\|_{L^{2}(|\xi|\le1)}d\tau\\
&\le C\int_{0}^{t/2}\sup_{|\xi|\le1}|\mathcal{F}[u^{2}-v^{2}](\xi, \tau)|\biggl(\int_{|\xi|\le1}|\xi|^{2(l+1)}e^{-2\mu(t-\tau)|\xi|^{2}}d\xi \biggl)^{1/2}d\tau\\
&\ \ \ +C\int_{t/2}^{t}\sup_{|\xi|\le1}|(i\xi)^{l}\mathcal{F}[u^{2}-v^{2}](\xi, \tau)|\biggl(\int_{|\xi|\le1}|\xi|^{2}e^{-2\mu(t-\tau)|\xi|^{2}}d\xi \biggl)^{1/2}d\tau\\
&\le C\int_{0}^{t/2}(1+t-\tau)^{-\frac{3}{4}-\frac{l}{2}}\|(u^{2}-v^{2})(\cdot, \tau)\|_{L^{1}}d\tau+C\int_{t/2}^{t}(1+t-\tau)^{-\frac{3}{4}}\|\p_{x}^{l}(u^{2}-v^{2})(\cdot, \tau)\|_{L^{1}}d\tau\\
&\le C(\|u\|_{X}+\|v\|_{X})\|u-v\|_{X}\\
&\ \ \ \times \biggl(\int_{0}^{t/2}(1+t-\tau)^{-\frac{3}{4}-\frac{l}{2}}(1+\tau)^{-\frac{1}{2}}d\tau+\int_{t/2}^{t}(1+t-\tau)^{-\frac{3}{4}}(1+\tau)^{-\frac{1}{2}-\frac{l}{2}}d\tau\biggl)\\
&\le C(\|u\|_{X}+\|v\|_{X})\|u-v\|_{X}(1+t)^{-\frac{1}{4}-\frac{l}{2}}.
\end{split}
\end{align}
For $|\xi|\ge1$, by using the Schwarz inequality, we have 
\begin{align*}
\begin{split}
|(i\xi)^{l}\hat{I}(\xi, t)|&=\biggl|(i\xi)^{l+1}\int_{0}^{t}e^{-\mu(t-\tau)|\xi|^{2}-\frac{i2Bb(t-\tau)\xi}{b^{2}+\xi^{2}}}\mathcal{F}[u^{2}-v^{2}](\xi, \tau)d\tau\biggl| \\
&\le C\int_{0}^{t}|\xi|e^{-\mu(t-\tau)|\xi|^{2}}|(i\xi)^{l}\mathcal{F}[u^{2}-v^{2}](\xi, \tau)|d\tau \\
&\le C\biggl(\int_{0}^{t}|\xi|^{2}e^{-\mu(t-\tau)|\xi|^{2}}d\tau \biggl)^{1/2}\biggl(\int_{0}^{t}e^{-\mu(t-\tau)|\xi|^{2}}|(i\xi)^{l}\mathcal{F}[u^{2}-v^{2}](\xi, \tau)|^{2}d\tau \biggl)^{1/2} \\
&\le C\biggl(\int_{0}^{t}e^{-\mu(t-\tau)|\xi|^{2}}|(i\xi)^{l}\mathcal{F}[u^{2}-v^{2}](\xi, \tau)|^{2}d\tau \biggl)^{1/2}.
\end{split}
\end{align*}
Therefore we have from \eqref{2-15}
\begin{align}
\label{2-19}
\begin{split}
I_{2}&\le C\biggl(\int_{|\xi|\ge1}\int_{0}^{t}e^{-\mu(t-\tau)|\xi|^{2}}|(i\xi)^{l}\mathcal{F}[u^{2}-v^{2}](\xi, \tau)|^{2}d\tau d\xi\biggl)^{1/2} \\
&\le C\biggl(\int_{0}^{t}e^{-\mu(t-\tau)}\int_{|\xi|\ge1}|(i\xi)^{l}\mathcal{F}[u^{2}-v^{2}](\xi, \tau)|^{2}d\xi d\tau \biggl)^{1/2} \\
&\le C\biggl(\int_{0}^{t}e^{-\mu(t-\tau)}\|\p_{x}^{l}(u^{2}-v^{2})(\cdot, \tau)\|_{L^{2}}^{2}d\tau \biggl)^{1/2} \\
&\le C(\|u\|_{X}+\|v\|_{X})\|u-v\|_{X}\biggl(\int_{0}^{t}e^{-\mu(t-\tau)}(1+\tau)^{-\frac{3}{2}-l}d\tau \biggl)^{1/2} \\
&\le C(\|u\|_{X}+\|v\|_{X})\|u-v\|_{X}(1+t)^{-\frac{3}{4}-\frac{l}{2}}. 
\end{split}
\end{align}
Combining \eqref{2-16} through \eqref{2-19}, we obtain 
\begin{equation*}
\|\p_{x}^{l}(N[u]-N[v])(t)\|_{L^{2}}\le C(\|u\|_{X}+\|v\|_{X})\|u-v\|_{X}(1+t)^{-\frac{1}{4}-\frac{l}{2}}, \ \ t\ge0
\end{equation*}
for $0\le l \le s$. Thus, there exists a positive constant $C_{1}>0$ such that 
\begin{equation*}
\|N[u]-N[v]\|_{X}\le C_{1}(\|u\|_{X}+\|v\|_{X})\|u-v\|_{X}\le 4C_{0}C_{1}E_{s, 0}\|u-v\|_{X}, \ \ u, v \in Y.
\end{equation*}
Choosing $E_{s, 0}$ which satisfies $4C_{0}C_{1}E_{s, 0} \le1/2$, then we have \eqref{2-12}. 
Moreover, taking $v=0$ in \eqref{2-12}, it follows that 
\begin{equation*}
\|N[u]-N[0]\|_{X}\le C_{0}E_{s, 0}.
\end{equation*}
Since $N[0]=N_{0}$, we obtain from \eqref{2-10} that 
\begin{equation*}
\|N[u]\|_{X}\le \|N_{0}\|_{X}+\|N[u]-N[0]\|_{X} \le 2C_{0}E_{s, 0}.
\end{equation*}
Therefore, we get \eqref{2-11}. This completes the proof of the global existence and of the $L^{2}$-decay estimate \eqref{2-5}. 
The $L^{\infty}$-estimate \eqref{2-6} immediately follows from the Sobolev inequality.  
\end{proof}

\section{Basic Estimates and Auxiliary Problem}

In this section, we prepare a couple of lemmas to prove the main theorems. First, we introduce some decay estimates for the asymptotic functions. 
Now, let us treat the nonlinear diffusion wave $\chi(x, t)$ defined by~\eqref{chi1}. For this function, it is easy to see that 
\begin{equation}
\label{chi-est}
|\chi(x, t)| \le C|M|(1+t)^{-\frac{1}{2}}e^{-\frac{(x-\alpha t)^{2}}{4\mu(1+t)}}, \ x\in \R, \ t\ge0.
\end{equation}
Moreover, $\chi(x, t)$ satisfies the following $L^{p}$-decay estimate (for the proof, see Lemma~4.3 in~\cite{KU17}).
\begin{lem}
\label{lem.chi-decay}
Let $k$, $l$ and $m$ be non-negative integers. Then, for $|M|\le1$ and $p\in[1, \infty]$, we have
\begin{equation}
\label{chi-decay}
\| \p_{t}^{k}\p_{x}^{l}(\p_{t}+\alpha \p_{x})^{m}\chi(\cdot, t)\|_{L^{p}}\le C|M|(1+t)^{-\frac{1}{2}+\frac{1}{2p}-\frac{k+l+2m}{2}}, \ \ t\ge0.
\end{equation}
\end{lem}
\noindent 
Next, we deal with the modified heat kernel 
\begin{equation}\label{heat}
G_{0}(x, t):=\frac{1}{\sqrt{4\pi \mu t}}e^{-\frac{(x-\alpha t)^{2}}{4\mu t}}, \ x\in \R, \ t>0, \ \alpha \in \R.
\end{equation}
This is a fundamental solution to the convection-heat equation $w_{t}+\alpha w_{x}=\mu w_{xx}$. Moreover, this function satisfies the following decay estimates (for the proof, see e.g. Lemma~7.1 in~\cite{UK07}): 
\begin{lem}
\label{lem.heat-decay}
Let $k$ and $l$ be non-negative integers. Then, for $p\in [1, \infty]$, we have
\begin{equation}
\label{heat-decay1} 
\| \p_{t}^{k}\p_{x}^{l}G_{0}(\cdot, t)\|_{L^{p}}\le Ct^{-\frac{1}{2}+\frac{1}{2p}-\frac{k+l}{2}}, \ \ t>0.
\end{equation}
Moreover, let $\phi \in L^{1}_{1}(\R)$ and suppose $\displaystyle \int_{\R}\phi(x)dx=0$, then we have
\begin{equation}
\label{heat-decay2}
\|\p_{t}^{k}\p_{x}^{l}G_{0}(t)*\phi\|_{L^{p}}\le Ct^{-\frac{1}{2}+\frac{1}{2p}-\frac{k+l}{2}}(1+t)^{-\frac{1}{2}}\|\phi\|_{L^{1}_{1}}, \ t>0. 
\end{equation}
\end{lem}
Next, for the latter sake, we introduce some estimates for $\eta(x, t)$ defined by \eqref{eta}. For this function, we can easily see that 
\begin{align}
&\min \{1, e^{\frac{\beta M}{2\mu}}\} \le \eta(x, t) \le \max \{1, e^{\frac{\beta M}{2\mu}}\}, \label{eta-est1}\\ 
&\min \{1, e^{-\frac{\beta M}{2\mu}}\} \le \eta(x, t)^{-1} \le \max \{1, e^{-\frac{\beta M}{2\mu}}\}. \label{eta-est2}
\end{align}
Moreover, by using Lemma~\ref{lem.chi-decay}, we have the following decay estimate (for the proof, see Corollary~2.3 in~\cite{K07} or Lemma~5.4 in~\cite{KU17}): 
\begin{lem}\label{lem.eta-decay}
Let $l$ be a positive integer and $p\in[1, \infty]$. If $|M| \le1$, then we have 
\begin{align}\label{eta-decay}
\| \p^{l}_{x}\eta(\cdot, t)\|_{L^{p}}+\| \p^{l}_{x}(\eta(\cdot, t)^{-1})\|_{L^{p}}&\le C|M|(1+t)^{-\frac{1}{2}+\frac{1}{2p}-\frac{l}{2}+\frac{1}{2}}, \ \ t\ge0.
\end{align}
\end{lem}

In order to prove Theorem~\ref{thm.main2} and Theorem~\ref{thm.main3}, we introduce an auxiliary problem. First, we set $\displaystyle \alpha=\frac{2B}{b}$ in \eqref{chi1} and  
\begin{equation}
\label{psi1}
\psi(x, t):=u(x, t)-\chi(x, t), \ \psi_{0}(x):=u_{0}(x)-\chi(x, 0), 
\end{equation}
where $u(x, t)$ is the solution to \eqref{VFW} and $\chi(x, t)$ is the nonlinear diffusion wave defined by \eqref{chi1} with $\displaystyle \alpha=\frac{2B}{b}$. Then, from~\eqref{VFW'} and~\eqref{BC}, $\psi(x, t)$ satisfies 
\begin{equation*}
\psi_{t}+(\beta \chi \psi)_{x}-\mu \psi_{xx}+2Bb(b^{2}-\p_{x}^{2})^{-1}u_{x}-\alpha \chi_{x}+\biggl(\frac{\beta}{2}\psi^{2}\biggl)_{x}=0.
\end{equation*} 
Moreover, since
\begin{align*}
2Bb(b^{2}-\p_{x}^{2})^{-1}u_{x}&=2Bb(b^{2}-\p_{x}^{2})^{-1}u_{x}-\alpha u_{x}+\alpha u_{x} \\
&=2Bb(b^{2}-\p_{x}^{2})^{-1}\left(u_{x}-(b^{2}-\p_{x}^{2})\frac{u_{x}}{b^{2}}\right)+\alpha u_{x}\\
&=\alpha(b^{2}-\p_{x}^{2})^{-1}\p_{x}^{3}u+\alpha u_{x}, 
\end{align*}
we have the following problem: 
\begin{align}
\begin{split}\label{psi.eq}
&\psi_{t}+\alpha \psi_{x}+(\beta \chi \psi)_{x}-\mu \psi_{xx}+\alpha(b^{2}-\p_{x}^{2})^{-1}\p_{x}^{3}u+\biggl(\frac{\beta}{2}\psi^{2}\biggl)_{x}=0, \ \ x\in \R, \ t>0,\\
&\psi(x, 0)=u_{0}(x)-\chi(x, 0)=\psi_{0}(x), \ \ x\in \R.  
\end{split}
\end{align}
To analyze the above problem, we prepare the following auxiliary problem:
\begin{align}
\label{aux}
\begin{split}
z_{t}+\alpha z_{x}+(\beta \chi z)_{x}-\mu z_{xx}&=\p_{x}\lambda(x, t), \ x\in \R, \ t>0, \\
z(x, 0)&=z_{0}(x) , \ \ x\in \R, 
\end{split}
\end{align}
where $\lambda(x, t)$ is a given regular function decaying at spatial infinity. If we set 
\begin{align}
\label{U}
\begin{split}
U[h](x, t, \tau):=\int_{\R}\p_{x}(G_{0}(x-y, t-\tau)\eta(x, t))(\eta(y, \tau))^{-1}\biggl(\int_{-\infty}^{y}h(\xi)d\xi\biggl)dy&,\\
x\in \R, \ 0\le \tau<t,& 
\end{split}
\end{align}
then, applying Lemma~2.6 in~\cite{FP} or Lemma~5.1 in~\cite{KU17}, we have the following formula: 
\begin{lem}\label{lem.formula}
Let $z_{0}(x)$ be a sufficiently regular function decaying at spatial infinity. Then we can get the smooth solution of \eqref{aux} which satisfies the following formula:
\begin{equation}
\label{formula}
z(x, t)=U[z_{0}](x, t, 0)+\int_{0}^{t}U[\p_{x}\lambda(\tau)](x, t, \tau)d\tau, \ x\in \R, \ t>0.
\end{equation}
\end{lem}
\noindent
This explicit representation formula~\eqref{formula} plays an important role in the proofs of the main theorems, especially in the proofs of Proposition~\ref{prop.second2} and Proposition~\ref{prop.third1} below. 

In the rest of this section, we prepare two useful estimates to prove Theorem~\ref{thm.main2} and Theorem~\ref{thm.main3}. 
First, for the linear part of the solution $z(x, t)$ in~\eqref{formula}, the following estimate is established (for the proof, see Corollary~3.4 in~\cite{K07}): 
\begin{lem}\label{lem.est-U1}
Let $s\ge1$. Assume that $|M| \le1$, $z_{0}\in H^{s}(\R)\cap L^{1}_{1}(\R)$ and $\displaystyle \int_{\R}z_{0}(x)dx=0$. Then the estimate 
\begin{equation}\label{est-U1}
\| \p^{l}_{x}U[z_{0}](\cdot, t, 0)\|_{L^{2}} \le CE_{s, 1}(1+t)^{-\frac{3}{4}-\frac{l}{2}}, \ \ t>0
\end{equation}
holds for any integer $0\le l \le s$.
\end{lem}
\noindent
Next, to evaluate the nonlocal dispersion terms, we derive the following estimate:
\begin{lem}\label{lem.est-dispersion}
Let $m$ be a non-negative integer and $p\in [1, \infty]$. Suppose $f\in W^{m,p}(\R)$. Then, the estimate 
\begin{equation}\label{est-dispersion}
\| (b^{2}-\p_{x}^{2})^{-1}\p^{l}_{x}f\|_{L^{p}} \le C\|\p_{x}^{l}f\|_{L^{p}}
\end{equation}
holds for any integer $0\le l \le m$.
\end{lem}
\begin{proof}
Since 
\begin{align*}
&(b^{2}-\p_{x}^{2})^{-1}\p_{x}^{l}f(x)=\mathcal{F}^{-1}\left[\frac{(i\xi)^{l}}{b^{2}+\xi^{2}}\hat{f}(\xi)\right](x)=\frac{1}{2b}\int_{\R}e^{-b|x-y|}\p_{y}^{l}f(y)dy=\frac{1}{2b}(e^{-b|\cdot|}*\p_{x}^{l}f)(x),
\end{align*}
the desired estimate \eqref{est-dispersion} immediately follows from Young's inequality. 
\end{proof}

\section{Asymptotic Behavior}

In this section, we shall show that the asymptotic profile of the solutions to~\eqref{VFW} is given by $\chi(x, t)$ defined by~\eqref{chi1} with $\displaystyle \alpha=\frac{2B}{b}$. Namely, the purpose of this section is to prove Theorem~\ref{thm.main1}. First, applying the Duhamel principle to~\eqref{BC}, we obtain the following integral equation:
\begin{equation}
\label{4-1}
\chi(t)=G_{0}(t)*\chi_{0}-\frac{\beta}{2}\int_{0}^{t}G_{0}(t-\tau)*((\chi^{2})_{x})(\tau)d\tau, 
\end{equation}
where $G_{0}(x, t)$ is defined by \eqref{heat} with $\displaystyle \alpha=\frac{2B}{b}$ and $\chi_{0}(x):=\chi(x, 0)$. Therefore, recalling
\begin{equation}\tag{\ref{psi1}}
\psi(x, t):=u(x, t)-\chi(x, t), \ \psi_{0}(x):=u_{0}(x)-\chi(x, 0), 
\end{equation}
combining \eqref{2-7} and \eqref{4-1}, we have
\begin{align}
\label{psi2}
\begin{split}
\psi(t)&=(T-G_{0})(t)*u_{0}+G_{0}(t)*(u_{0}-\chi_{0})\\
&\ \ \ \ -\frac{\beta}{2}\int_{0}^{t}(T-G_{0})(t-\tau)*((u^{2})_{x})(\tau)d\tau-\frac{\beta}{2}\int_{0}^{t}G_{0}(t-\tau)*((u^{2}-\chi^{2})_{x})(\tau)d\tau\\
&=:K_{1}+K_{2}+K_{3}+K_{4}. 
\end{split}
\end{align}

Our first step to prove Theorem~\ref{thm.main1} is to derive the following asymptotic relation: 
\begin{lem}
\label{lem.asymp1}
Let $l$ be a non-negative integer. Then, for $p\in [2, \infty]$, we have the following estimate: 
\begin{equation}
\label{asymp1}
\|\p_{x}^{l}(T(\cdot, t)-G_{0}(\cdot, t))\|_{L^{p}}\le Ct^{-1+\frac{1}{2p}-\frac{l}{2}}, \ \ t>0,
\end{equation}
where $T(x, t)$ and $G_{0}(x, t)$ are defined by \eqref{2-1} and \eqref{heat} with $\displaystyle \alpha=\frac{2B}{b}$, respectively. 
\end{lem}
\begin{proof}
We set $\displaystyle \alpha=\frac{2B}{b}$ in \eqref{heat}. By Plancherel's theorem and the definitions of $T(x, t)$ and $G_{0}(x, t)$, we have
\begin{align*}
\begin{split}
\|\p_{x}^{l}(T(\cdot, t)-G_{0}(\cdot, t))\|_{L^{2}}^{2}&=\left\|(i\xi)^{l}(e^{-\mu t|\xi|^{2}-\frac{i2Bbt\xi}{b^{2}+\xi^{2}}}-e^{-\mu t|\xi|^{2}-\frac{i2Bt\xi}{b}})\right\|_{L^{2}}^{2}\\
&=\int_{\R}|\xi|^{2l}e^{-2\mu t|\xi|^{2}}\left|e^{-\frac{i2Bbt\xi}{b^{2}+\xi^{2}}}-e^{-\frac{i2Bt\xi}{b}}\right|^{2}d\xi=:J(t). 
\end{split}
\end{align*}
From the mean value theorem, there exists $\theta=\theta(\xi, t, B, b)$ such that 
\begin{align*}
\begin{split}
e^{-\frac{i2Bbt\xi}{b^{2}+\xi^{2}}}-e^{-\frac{i2Bt\xi}{b}}&=e^{i\theta}i\left(\frac{2Bt\xi}{b}-\frac{2Bbt\xi}{b^{2}+\xi^{2}}\right)=\frac{2e^{i\theta}iBt}{b}\frac{\xi^{3}}{b^{2}+\xi^{2}}.
\end{split}
\end{align*}
Thus we obtain 
\begin{align*}
J(t)&=\left(\frac{2Bt}{b}\right)^{2}\int_{\R}|\xi|^{2l}e^{-2\mu t|\xi|^{2}}\left|\frac{\xi^{3}}{b^{2}+\xi^{2}}\right|^{2}d\xi \le Ct^{2}\int_{\R}|\xi|^{2(l+3)}e^{-2\mu t|\xi|^{2}}d\xi \le Ct^{-\frac{3}{2}-l}. 
\end{align*}
Therefore, we have the $L^{2}$-estimate
 \begin{equation}
 \label{4-2}
\|\p_{x}^{l}(T(\cdot, t)-G_{0}(\cdot, t))\|_{L^{2}}\le Ct^{-\frac{3}{4}-\frac{l}{2}}, \ \ t>0. 
\end{equation}
From the Sobolev inequality, we see that 
 \begin{align}
 \label{4-3}
 \begin{split}
&\|\p_{x}^{l}(T(\cdot, t)-G_{0}(\cdot, t))\|_{L^{\infty}}\\
&\le \|\p_{x}^{l}(T(\cdot, t)-G_{0}(\cdot, t))\|_{L^{2}}^{1/2}\|\p_{x}^{l+1}(T(\cdot, t)-G_{0}(\cdot, t))\|_{L^{2}}^{1/2}\le Ct^{-1-\frac{l}{2}}, \ \ t>0. 
\end{split}
\end{align}
Finally for $p\in [2, \infty]$, by the interpolation inequality, combining \eqref{4-2} and \eqref{4-3}, we arrive at  
\begin{align*}
 \begin{split}
&\|\p_{x}^{l}(T(\cdot, t)-G_{0}(\cdot, t))\|_{L^{p}}\\
&\le \|\p_{x}^{l}(T(\cdot, t)-G_{0}(\cdot, t))\|_{L^{\infty}}^{1-2/p}\|\p_{x}^{l}(T(\cdot, t)-G_{0}(\cdot, t))\|_{L^{2}}^{2/p}\le Ct^{-1+\frac{1}{2p}-\frac{l}{2}}, \ \ t>0. 
\end{split}
\end{align*}
This completes the proof. 
\end{proof}

Next, we shall derive the $L^{2}$-decay estimate of $\psi(x, t)$:
\begin{prop}
\label{prop.psi-L2}
Assume the same conditions on $u_{0}$ in Theorem~\ref{thm.main1} are valid. Then, for any $\e>0$, we have 
\begin{align}
\label{psi-L2}
\|\p_{x}^{l}\psi(\cdot, t)\|_{L^{2}}\le CE_{s, 1}(1+t)^{-\frac{3}{4}-\frac{l}{2}+\e}, \ \ t\ge0
\end{align}
for any integer $0\le l\le s-1$, where $\psi(x, t)$ is defined by \eqref{psi1}.
\end{prop} 
\begin{proof}
We set 
\begin{align}
\label{N}
N(T):=\sup_{0\le t\le T}\sum_{l=0}^{s-1}(1+t)^{\frac{3}{4}+\frac{l}{2}-\e}\|\p_{x}^{l}\psi(\cdot, t)\|_{L^{2}}, 
\end{align}
where $\e$ is any fixed constant such that $0<\e<\frac{3}{4}$. Then, let us evaluate each term of the right hand side of \eqref{psi2}. First for $K_{1}$, from Young's inequality and the above Lemma \ref{lem.asymp1}, we have
\begin{align}
\label{4-4}
\begin{split}
\|\p_{x}^{l}K_{1}(\cdot, t)\|_{L^{2}}&\le C\|u_{0}\|_{L^{1}}(1+t)^{-\frac{3}{4}-\frac{l}{2}}, \ t\ge1.
\end{split}
\end{align}
By the assumptions on the initial data, \eqref{chi2} and \eqref{chi-est}, we get $\psi_{0}\in L^{1}_{1}(\R)$. Therefore, applying \eqref{heat-decay2} to $K_{2}$, it follows that 
\begin{align}
\label{4-5}
\begin{split}
\|\p_{x}^{l}K_{2}(\cdot, t)\|_{L^{2}}&\le C\|u_{0}\|_{L^{1}_{1}}(1+t)^{-\frac{3}{4}-\frac{l}{2}}, \ t\ge1.
\end{split}
\end{align}
Next, we evaluate $K_{3}$ and $K_{4}$. Before do that, we prepare the following estimates for $0\le l \le s-1$:  
\begin{align}
\label{4-6}
\|\p_{x}^{l+1}(u^{2}(\cdot, t))\|_{L^{1}}&\le CE_{s, 0}(1+t)^{-1-\frac{l}{2}}, \\
\label{4-7}
\|\p_{x}^{l}((u^{2}-\chi^{2})(\cdot, t))\|_{L^{1}}&\le CE_{s, 0}N(T)(1+t)^{-1-\frac{l}{2}+\e}.
\end{align}
Let $0\le l \le s-1$ and $0\le t \le T$. From \eqref{2-5}, \eqref{chi-decay} and \eqref{N}, we have 
\begin{align*}
\begin{split}
\|\p_{x}^{l+1}(u^{2}(\cdot, t))\|_{L^{1}}
&\le C\sum_{m=0}^{l+1}\|\p_{x}^{m}u(\cdot, t)\|_{L^{2}}\|\p_{x}^{l+1-m}u(\cdot, t)\|_{L^{2}} \\
&\le CE_{s, 0}(1+t)^{-1-\frac{l}{2}}  
\end{split}
\end{align*}
and 
\begin{align*}
\begin{split}
\|\p_{x}^{l}((u^{2}-\chi^{2})(\cdot, t))\|_{L^{1}}&=\|\p_{x}^{l}((\psi(u+\chi))(\cdot, t))\|_{L^{1}}\\
&\le C\sum_{m=0}^{l}\|\p_{x}^{m}\psi(\cdot, t)\|_{L^{2}}\|\p_{x}^{l-m}((u+\chi)(\cdot, t))\|_{L^{2}} \\
&\le CE_{s, 0}N(T)\sum_{m=0}^{l}(1+t)^{-\frac{3}{4}-\frac{m}{2}+\e}(1+t)^{-\frac{1}{4}-\frac{l}{2}+\frac{m}{2}}\\
&\le CE_{s, 0}N(T)(1+t)^{-1-\frac{l}{2}+\e}.  
\end{split}
\end{align*}
Therefore, by using Young's inequality, Lemma~\ref{lem.asymp1} and~\eqref{4-6}, we obtain 
\begin{align}
\label{4-8}
\begin{split}
\|\p_{x}^{l}K_{3}(\cdot, t)\|_{L^{2}}&\le C\int_{0}^{t/2}\|\p_{x}^{l+1}(T-G_{0})(t-\tau)*u^{2}(\tau)\|_{L^{2}}d\tau\\
&\ \ \ +C\int_{t/2}^{t}\|(T-G_{0})(t-\tau)*\p_{x}^{l+1}(u^{2})(\tau)\|_{L^{2}}d\tau\\
&\le C\int_{0}^{t/2}\|\p_{x}^{l+1}(T-G_{0})(\cdot, t-\tau)\|_{L^{2}}\|u^{2}(\cdot, \tau)\|_{L^{1}}d\tau\\
&\ \ \ +C\int_{t/2}^{t}\|(T-G_{0})(\cdot, t-\tau)\|_{L^{2}}\|\p_{x}^{l+1}(u^{2}(\cdot, \tau))\|_{L^{1}}d\tau\\
&\le CE_{s, 0}\int_{0}^{t/2}(t-\tau)^{-\frac{5}{4}-\frac{l}{2}}(1+\tau)^{-\frac{1}{2}}d\tau+CE_{s, 0}\int_{t/2}^{t}(t-\tau)^{-\frac{3}{4}}(1+\tau)^{-1-\frac{l}{2}}d\tau\\
&\le CE_{s, 0}(1+t)^{-\frac{3}{4}-\frac{l}{2}}, \ t\ge1. 
\end{split}
\end{align}
On the other hand, from Young's inequality, \eqref{heat-decay1} and~\eqref{4-7}, it follows that 
\begin{align}
\label{4-9}
\begin{split}
\|\p_{x}^{l}K_{4}(\cdot, t)\|_{L^{2}}&\le C\int_{0}^{t/2}\|\p_{x}^{l+1}G_{0}(t-\tau)*(u^{2}-\chi^{2})(\tau)\|_{L^{2}}d\tau\\
&\ \ \ +C\int_{t/2}^{t}\|\p_{x}G_{0}(t-\tau)*\p_{x}^{l}(u^{2}-\chi^{2})(\tau)\|_{L^{2}}d\tau\\
&\le C\int_{0}^{t/2}\|\p_{x}^{l+1}G_{0}(\cdot, t-\tau)\|_{L^{2}}\|(u^{2}-\chi^{2})(\cdot, \tau)\|_{L^{1}}d\tau\\
&\ \ \ +C\int_{t/2}^{t}\|\p_{x}G_{0}(\cdot, t-\tau)\|_{L^{2}}\|\p_{x}^{l}((u^{2}-\chi^{2})(\cdot, \tau))\|_{L^{1}}d\tau\\
&\le CE_{s, 0}N(T)\int_{0}^{t/2}(t-\tau)^{-\frac{3}{4}-\frac{l}{2}}(1+\tau)^{-1+\e}d\tau\\
&\ \ \ +CE_{s, 0}N(T)\int_{t/2}^{t}(t-\tau)^{-\frac{3}{4}}(1+\tau)^{-1-\frac{l}{2}+\e}d\tau\\
&\le CE_{s, 0}N(T)(1+t)^{-\frac{3}{4}-\frac{l}{2}+\e}, \ t\ge1. 
\end{split}
\end{align}
Thus, combining \eqref{4-4}, \eqref{4-5}, \eqref{4-8} and \eqref{4-9}, we arrive at  
\begin{align}
\label{4-10}
\begin{split}
\|\p_{x}^{l}\psi(\cdot, t)\|_{L^{2}}\le CE_{s, 1}(1+t)^{-\frac{3}{4}-\frac{l}{2}}+CE_{s, 0}N(T)(1+t)^{-\frac{3}{4}-\frac{l}{2}+\e}, \ 1\le t\le T. 
\end{split}
\end{align}
For $0\le t\le 1$, from \eqref{2-5}, \eqref{chi-decay} and $|M|\le E_{s, 0}\le E_{s, 1}$, we easily see  
\begin{equation}
\label{4-11}
\|\p_{x}^{l}\psi(\cdot, t)\|_{L^{2}}\le \|\p_{x}^{l}u(\cdot, t)\|_{L^{2}}+\|\p_{x}^{l}\chi(\cdot, t)\|_{L^{2}}\le CE_{s, 1}, \ 0\le t\le 1.
\end{equation}
Summing up \eqref{4-10} and \eqref{4-11}, it follows that  
\begin{equation*}
N(T)\le CE_{s, 1}+C_{1}E_{s, 0}N(T),
\end{equation*}
where $C_{1}$ is a positive constant. Therefore, we arrive at the desired estimate 
\begin{equation*}
N(T)\le 2CE_{s, 1}
\end{equation*}
if $E_{s, 0}$ is so small that $C_{1}E_{s, 0}\le\frac{1}{2}$. This completes the proof. 
\end{proof}

Finally, we shall derive the $L^{\infty}$-decay estimate of $\psi(x, t)$:
\begin{prop}
\label{prop.psi-Linf}
Assume the same conditions on $u_{0}$ in Theorem~\ref{thm.main1} are valid. Then, for any $\e>0$, we have 
\begin{align}
\label{psi-Linf}
\|\p_{x}^{l}\psi(\cdot, t)\|_{L^{\infty}}\le CE_{s, 1}(1+t)^{-1-\frac{l}{2}+\e}, \ \ t\ge0
\end{align}
for any integer $0\le l\le s-1$, where $\psi(x, t)$ is defined by~\eqref{psi1}.
\end{prop} 
\begin{proof}
We evaluate each term of the right hand side of \eqref{psi2}. For $K_{1}$, from Young's inequality and Lemma \ref{lem.asymp1}, we have
\begin{align}
\label{4-12}
\begin{split}
\|\p_{x}^{l}K_{1}(\cdot, t)\|_{L^{\infty}}&\le C\|u_{0}\|_{L^{1}}(1+t)^{-1-\frac{l}{2}}, \ t\ge1.
\end{split}
\end{align}
In the same way to get \eqref{4-5}, we obtain from \eqref{heat-decay2}
\begin{align}
\label{4-13}
\begin{split}
\|\p_{x}^{l}K_{2}(\cdot, t)\|_{L^{\infty}}&\le C\|u_{0}\|_{L^{1}_{1}}(1+t)^{-1-\frac{l}{2}}, \ t\ge1.
\end{split}
\end{align}
Next, we evaluate $K_{3}$ and $K_{4}$. Similarly as \eqref{4-6} and \eqref{4-7}, using \eqref{2-5}, \eqref{2-6}, \eqref{chi-decay} and \eqref{psi-L2}, we have the following estimate for $0\le l \le s-1$:  
\begin{align}
\label{4-14}
\|\p_{x}^{l+1}(u^{2}(\cdot, t))\|_{L^{2}}&\le CE_{s, 0}(1+t)^{-\frac{5}{4}-\frac{l}{2}}, \\
\label{4-15}
\|\p_{x}^{l}((u^{2}-\chi^{2})(\cdot, t))\|_{L^{2}}&\le CE_{s, 1}(1+t)^{-\frac{5}{4}-\frac{l}{2}+\e}. 
\end{align}
By using Young's inequality, Lemma~\ref{lem.asymp1}, \eqref{4-6} and \eqref{4-14}, we obtain 
\begin{align}
\label{4-16}
\begin{split}
\|\p_{x}^{l}K_{3}(\cdot, t)\|_{L^{\infty}}&\le C\int_{0}^{t/2}\|\p_{x}^{l+1}(T-G_{0})(t-\tau)*u^{2}(\tau)\|_{L^{\infty}}d\tau\\
&\ \ \ +C\int_{t/2}^{t}\|(T-G_{0})(t-\tau)*\p_{x}^{l+1}(u^{2})(\tau)\|_{L^{\infty}}d\tau\\
&\le C\int_{0}^{t/2}\|\p_{x}^{l+1}(T-G_{0})(\cdot, t-\tau)\|_{L^{\infty}}\|u^{2}(\cdot, \tau)\|_{L^{1}}d\tau\\
&\ \ \ +C\int_{t/2}^{t}\|(T-G_{0})(\cdot, t-\tau)\|_{L^{2}}\|\p_{x}^{l+1}(u^{2}(\cdot, \tau))\|_{L^{2}}d\tau\\
&\le CE_{s, 0}\int_{0}^{t/2}(t-\tau)^{-\frac{3}{2}-\frac{l}{2}}(1+\tau)^{-\frac{1}{2}}d\tau+CE_{s, 0}\int_{t/2}^{t}(t-\tau)^{-\frac{3}{4}}(1+\tau)^{-\frac{5}{4}-\frac{l}{2}}d\tau\\
&\le CE_{s, 0}(1+t)^{-1-\frac{l}{2}}, \ t\ge1. 
\end{split}
\end{align}
Moreover, from Young's inequality, \eqref{heat-decay1}, \eqref{4-7} and \eqref{4-15}, it follows that 
\begin{align}
\label{4-17}
\begin{split}
\|\p_{x}^{l}K_{4}(\cdot, t)\|_{L^{\infty}}&\le C\int_{0}^{t/2}\|\p_{x}^{l+1}G_{0}(t-\tau)*(u^{2}-\chi^{2})(\tau)\|_{L^{\infty}}d\tau\\
&\ \ \ +C\int_{t/2}^{t}\|\p_{x}G_{0}(t-\tau)*\p_{x}^{l}(u^{2}-\chi^{2})(\tau)\|_{L^{\infty}}d\tau\\
&\le C\int_{0}^{t/2}\|\p_{x}^{l+1}G_{0}(\cdot, t-\tau)\|_{L^{\infty}}\|(u^{2}-\chi^{2})(\cdot, \tau)\|_{L^{1}}d\tau\\
&\ \ \ +C\int_{t/2}^{t}\|\p_{x}G_{0}(\cdot, t-\tau)\|_{L^{2}}\|\p_{x}^{l}((u^{2}-\chi^{2})(\cdot, \tau))\|_{L^{2}}d\tau\\
&\le CE_{s, 1}\int_{0}^{t/2}(t-\tau)^{-1-\frac{l}{2}}(1+\tau)^{-1+\e}d\tau+CE_{s, 1}\int_{t/2}^{t}(t-\tau)^{-\frac{3}{4}}(1+\tau)^{-\frac{5}{4}-\frac{l}{2}+\e}d\tau\\
&\le CE_{s, 1}(1+t)^{-1-\frac{l}{2}+\e}, \ t\ge1. 
\end{split}
\end{align}
Thus, combining \eqref{4-12}, \eqref{4-13}, \eqref{4-16} and \eqref{4-17}, we arrive at  
\begin{align}
\label{4-18}
\begin{split}
\|\p_{x}^{l}\psi(\cdot, t)\|_{L^{\infty}}&\le CE_{s, 1}(1+t)^{-1-\frac{l}{2}+\e}, \ t\ge1. 
\end{split}
\end{align}
For $0\le t\le 1$, from \eqref{2-6} and \eqref{chi-decay}, similarly as \eqref{4-11}, we get 
\begin{equation}
\label{4-19}
\|\p_{x}^{l}\psi(\cdot, t)\|_{L^{\infty}}\le \|\p_{x}^{l}u(\cdot, t)\|_{L^{\infty}}+\|\p_{x}^{l}\chi(\cdot, t)\|_{L^{\infty}}\le CE_{s, 1}, \ 0\le t\le 1.
\end{equation}
Summing up \eqref{4-18} and \eqref{4-19}, we complete the proof. 
\end{proof}

\begin{proof}[\rm{\bf{Proof of Theorem~\ref{thm.main1}}}]
By the interpolation inequality, Proposition~\ref{prop.psi-L2} and Proposition~\ref{prop.psi-Linf}, we immediately obtain the estimate~\eqref{main1}. This completes the proof.
\end{proof}

\section{Second Asymptotic Profile}

Next in this section, we shall prove our second main result Theorem~\ref{thm.main2}. First, we consider 
\begin{align}
\begin{split}\label{v}
v_{t}+\alpha v_{x}+(\beta \chi v)_{x}-\mu v_{xx}+\frac{2B}{b^{3}}\chi_{xxx}&=0 , \ \ x\in \R, \ \ t>0, \\
v(x, 0)&=0, \ \ x\in \R.
\end{split}
\end{align}
The leading term of the solution $v(x, t)$ to \eqref{v} is given by $V(x, t)$ defined by \eqref{second1}. Actually, by the change of variable, from Proposition~4.3 in~\cite{F19-1}, we can easily obtain the following proposition:
\begin{prop}\label{prop.second1}
Let $l$ be a non-negative integer. Then, for $|M| \le 1$ and $p\in [1, \infty]$, we have 
\begin{equation}
\label{v-V}
\|\p_{x}^{l}(v(\cdot, t)-V(\cdot, t))\|_{L^{p}} \le C|M|(1+t)^{-1+\frac{1}{2p}-\frac{l}{2}}, \ t\ge1, 
\end{equation}
where $v(x, t)$ is the solution to \eqref{v}, while $V(x, t)$ is defined by \eqref{second1}.
\end{prop} 
By virtue of Proposition~\ref{prop.second1}, it is sufficient for the proof of Theorem~\ref{thm.main2} to show the following proposition: 
\begin{prop}
\label{prop.second2}
Let $s\ge2$. Assume that $u_{0}\in L^{1}_{1}(\R)\cap H^{s}(\R)$ and $E_{s, 0}$ is sufficiently small. Then, for the solution to~\eqref{VFW}, the estimate
\begin{equation}
\label{asymp2}
\|\p_{x}^{l}(u(\cdot, t)-\chi(\cdot, t)-v(\cdot, t))\|_{L^{p}}\le CE_{s, 1}(1+t)^{-1+\frac{1}{2p}-\frac{l}{2}}, \ t\ge0
\end{equation}
holds for any $p\in [2, \infty]$ and integer $0\le l\le s-2$, where $\chi(x, t)$ is defined by~\eqref{chi1} with $\displaystyle \alpha=\frac{2B}{b}$, while $v(x, t)$ is the solution to~\eqref{v}.
\end{prop}
\begin{proof}
Throughout this proof, we set $\displaystyle \alpha=\frac{2B}{b}$ and  
\begin{align}
\begin{split}\label{w.eq1}
w(x, t):=&\ u(x, t)-\chi(x, t)-v(x, t)=\psi(x, t)-v(x, t).
\end{split}
\end{align}
Then, from \eqref{psi.eq} and \eqref{v}, $w(x, t)$ satisfies 
\begin{equation*}
w_{t}+\alpha w_{x}+(\beta \chi w)_{x}-\mu w_{xx}=-\alpha(b^{2}-\p_{x}^{2})^{-1}\p_{x}^{3}u-\biggl(\frac{\beta}{2}\psi^{2}\biggl)_{x}+\frac{\alpha}{b^{2}}\chi_{xxx}.
\end{equation*}
Moreover, since 
\begin{align*}
\alpha(b^{2}-\p_{x}^{2})^{-1}\p_{x}^{3}u-\frac{\alpha}{b^{2}}\chi_{xxx}&=\alpha(b^{2}-\p_{x}^{2})^{-1}(\p_{x}^{3}u-\p_{x}^{3}\chi)+\alpha(b^{2}-\p_{x}^{2})^{-1}\p_{x}^{3}\chi-\frac{\alpha}{b^{2}}\chi_{xxx} \\
&=\alpha(b^{2}-\p_{x}^{2})^{-1}\p_{x}^{3}\psi+\alpha(b^{2}-\p_{x}^{2})^{-1}\left(\p_{x}^{3}\chi-(b^{2}-\p_{x}^{2})\frac{\p_{x}^{3}\chi}{b^{2}}\right)\\
&=\alpha(b^{2}-\p_{x}^{2})^{-1}\p_{x}^{3}\psi+\frac{\alpha}{b^{2}}(b^{2}-\p_{x}^{2})^{-1}\p_{x}^{5}\chi, 
\end{align*}
we obtain the following equation: 
\begin{align}
\begin{split}\label{w.eq2}
&w_{t}+\alpha w_{x}+(\beta \chi w)_{x}-\mu w_{xx}\\
&=-\alpha(b^{2}-\p_{x}^{2})^{-1}\p_{x}^{3}\psi-\frac{\alpha}{b^{2}}(b^{2}-\p_{x}^{2})^{-1}\p_{x}^{5}\chi-\biggl(\frac{\beta}{2}\psi^{2}\biggl)_{x}, \ \ x\in \R, \ t>0,\\
&w(x, 0)=u_{0}(x)-\chi(x, 0)=\psi_{0}(x), \ \ x\in \R.  
\end{split}
\end{align}
By the assumption of the initial data and \eqref{chi-est}, we get $\psi_{0}\in L^{1}_{1}(\R)\cap H^{s}(\R)$. Also, from \eqref{chi2} and \eqref{BC} we have $\displaystyle \int_{\R}\psi_{0}(x)dx=0$. From Lemma~\ref{formula}, we obtain the following integral equation: 
\begin{align}
\begin{split}\label{w.eq3}
w(x, t)&=U[\psi_{0}](x, t, 0)-\frac{\beta}{2}\int_{0}^{t}U[\p_{x}(\psi^{2})(\tau)](x, t, \tau)d\tau \\
&\ \ \ -\alpha\int_{0}^{t}U\left[(b^{2}-\p_{x}^{2})^{-1}\p_{x}^{3}\psi(\tau)\right](x, t, \tau)d\tau-\frac{\alpha}{b^{2}}\int_{0}^{t}U\left[(b^{2}-\p_{x}^{2})^{-1}\p_{x}^{5}\chi(\tau)\right](x, t, \tau)d\tau\\
&=:I_{1}+I_{2}+I_{3}+I_{4}.
\end{split}
\end{align}

We shall evaluate $I_{1}$, $I_{2}$, $I_{3}$ and $I_{4}$. First for $I_{1}$, by using Lemma~\ref{lem.est-U1} and $|M|\le \|u_{0}\|_{L^{1}_{1}}$, we get 
\begin{equation}\label{5-7}
\| \p_{x}^{l}I_{1}(\cdot, t)\|_{L^{2}} \le CE_{s, 1}(1+t)^{-\frac{3}{4}-\frac{l}{2}}, \ \ 0\le l\le s.
\end{equation}
From the Sobolev inequality, we see that 
\begin{align}
\begin{split}\label{5-8}
\| \p_{x}^{l}I_{1}(\cdot, t)\|_{L^{\infty}} &\le \sqrt{2} \| \p_{x}^{l}I_{1}(\cdot, t)\|_{L^{2}}^{1/2}\| \p_{x}^{l+1}I_{1}(\cdot, t)\|_{L^{2}}^{1/2}\\
&\le CE_{s, 1}(1+t)^{-1-\frac{l}{2}}, \ \ 0\le l\le s-1.
\end{split}
\end{align}
Therefore, for $p\in [2, \infty]$, by the interpolation inequality, combining \eqref{5-7} and \eqref{5-8}, we obtain
\begin{equation}\label{5-9}
\| \p_{x}^{l}I_{1}(\cdot, t)\|_{L^{p}} \le CE_{s, 1}(1+t)^{-1-\frac{1}{2p}-\frac{l}{2}}, \ \ 0\le l\le s-1.
\end{equation}

Next, we evaluate $I_{2}$, $I_{3}$ and $I_{4}$. In the following, let $p \in [2, \infty]$. First, for any given regular function $\lambda(x, t)$ and any integer $l$, from \eqref{U}, it follows that 
\begin{align*}
\p_{x}^{l}U[\p_{x}\lambda(\tau)](x, t, \tau)=\sum_{j=0}^{l+1}\begin{pmatrix}l+1 \\ j \end{pmatrix}\p_{x}^{l+1-j}\eta(x, t)\int_{\R}\p_{x}^{j}G_{0}(x-y, t-\tau)(\eta(y, \tau))^{-1}\lambda(y, \tau)dy. 
\end{align*}
Therefore, we have from Lemma~\ref{lem.eta-decay} and \eqref{eta-est1}
\begin{equation}\label{5-10}
\|\p_{x}^{l}U[\p_{x}\lambda(\tau)](\cdot, t, \tau)\|_{L^{p}}\le C\sum_{j=0}^{l+1}(1+t)^{-\frac{1}{2}(l+1-j)}\|\p_{x}^{j}J[\lambda](\cdot, t, \tau)\|_{L^{p}}, 
\end{equation}
where
\begin{equation}\label{5-11}
J[\lambda](x, t, \tau):=\int_{\R}G_{0}(x-y, t-\tau)(\eta(y, \tau))^{-1}\lambda(y, \tau)dy=(G_{0}(t-\tau)*(\eta^{-1}\lambda)(\tau))(x).
\end{equation}
By using \eqref{5-10} and \eqref{5-11}, we have
\begin{align}
\begin{split}\label{5-12}
\|\p_{x}^{l}I_{2}(\cdot, t)\|_{L^{p}}&\le C\sum_{j=0}^{l+1}(1+t)^{-\frac{1}{2}(l+1-j)}\int_{0}^{t}\|\p_{x}^{j}J[\psi^{2}](\cdot, t, \tau)\|_{L^{p}}d\tau\\
&\le C\sum_{j=0}^{l+1}(1+t)^{-\frac{1}{2}(l+1-j)}\left(\int_{0}^{t/2}+\int_{t/2}^{t}\right)\|\p_{x}^{j}J[\psi^{2}](\cdot, t, \tau)\|_{L^{p}}d\tau\\
&=:C\sum_{j=0}^{l+1}(1+t)^{-\frac{1}{2}(l+1-j)}(I_{2.1}+I_{2.2}).
\end{split}
\end{align}
For $I_{2.1}$, from \eqref{5-11}, Young's inequality, \eqref{eta-est2}, \eqref{heat-decay1} and Theorem~\ref{thm.main1}, we obtain 
\begin{align}
\begin{split}\label{5-13}
I_{2.1}&\le \int_{0}^{t/2}\|\p_{x}^{j}G_{0}(\cdot, t-\tau)\|_{L^{p}}\|\psi^{2}(\cdot, \tau)\|_{L^{1}}d\tau \\
&\le CE_{s, 1}\int_{0}^{t/2}(t-\tau)^{-\frac{1}{2}+\frac{1}{2p}-\frac{j}{2}}(1+\tau)^{-\frac{3}{2}+2\e}d\tau \le CE_{s, 1}(1+t)^{-\frac{1}{2}+\frac{1}{2p}-\frac{j}{2}}, \ \ t\ge1. 
\end{split}
\end{align}
On the other hand, for $I_{2.2}$ with $j=0$, we have from \eqref{heat-decay1} and Theorem~\ref{thm.main1}
\begin{align}
\begin{split}\label{5-14}
I_{2.2}&\le \int_{t/2}^{t}\|G_{0}(\cdot, t-\tau)\|_{L^{1}}\|\psi^{2}(\cdot, \tau)\|_{L^{p}}d\tau \le C\int_{t/2}^{t}\|\psi(\cdot, \tau)\|_{L^{\infty}}\|\psi(\cdot, \tau)\|_{L^{p}}d\tau\\
&\le CE_{s, 1}\int_{t/2}^{t}(1+\tau)^{-1+\e}(1+\tau)^{-1+\frac{1}{2p}+\e}d\tau \le CE_{s, 1}(1+t)^{-1+\frac{1}{2p}+2\e}, \ \ t\ge0. 
\end{split}
\end{align}
For $j=1, \cdots, l+1$, we prepare the following estimate
\begin{align*}
&\|\p_{x}^{j-1}((\eta^{-1}\psi^{2})(\cdot, \tau))\|_{L^{p}}\\
& \le C\sum_{m=0}^{j-1}\sum_{n=0}^{j-1-m}\|\p_{x}^{m}(\eta(\cdot, \tau)^{-1})\|_{L^{\infty}}\|\p_{x}^{n}\psi(\cdot, \tau)\|_{L^{\infty}}\|\p_{x}^{j-1-m-n}\psi(\cdot, \tau)\|_{L^{p}}\\
&\le CE_{s, 1}\sum_{m=0}^{j-1}\sum_{n=0}^{j-1-m}(1+\tau)^{-\frac{m}{2}}(1+\tau)^{-1-\frac{n}{2}+\e}(1+\tau)^{-1+\frac{1}{2p}-\frac{1}{2}(j-1-m-n)+\e}\\
&\le CE_{s, 1}(1+\tau)^{-2+\frac{1}{2p}-\frac{j-1}{2}+2\e}, 
\end{align*}
where we have used \eqref{lem.eta-decay}, \eqref{eta-est2} and Theorem~\ref{thm.main1}. Therefore, we get the following estimate for $I_{2.2}$ with $j=1, \cdots, l+1$: 
\begin{align}
\begin{split}\label{5-15}
I_{2.2}&\le \int_{t/2}^{t}\|\p_{x}G_{0}(\cdot, t-\tau)\|_{L^{1}}\|\p_{x}^{j-1}((\eta^{-1}\psi^{2})(\cdot, \tau))\|_{L^{p}}d\tau\\
&\le CE_{s, 1}\int_{t/2}^{t}(t-\tau)^{-\frac{1}{2}}(1+\tau)^{-2+\frac{1}{2p}-\frac{j-1}{2}+2\e}d\tau \le CE_{s, 1}(1+t)^{-1+\frac{1}{2p}-\frac{j}{2}+2\e}, \ \ t\ge0. 
\end{split}
\end{align}
Combining \eqref{5-14} and \eqref{5-15}, for all $j=0, 1, \cdots, l+1$, we have
\begin{equation}\label{5-16}
I_{2.2}\le CE_{s, 1}(1+t)^{-1+\frac{1}{2p}-\frac{j}{2}+2\e}, \ \ t\ge0.
\end{equation}
Therefore, summing up \eqref{5-12}, \eqref{5-13} and \eqref{5-16}, we arrive at 
\begin{equation}\label{5-17}
\|\p_{x}^{l}I_{2}(\cdot, t)\|_{L^{p}}\le CE_{s, 1}(1+t)^{-1+\frac{1}{2p}-\frac{l}{2}}, \ \ t\ge1, \ \ 0\le l \le s-1.
\end{equation}

Next, we deal with $I_{3}$. In the following, let $0\le l \le s-2$. By using \eqref{5-10} and \eqref{5-11}, we have
\begin{align}
\begin{split}\label{5-18}
\|\p_{x}^{l}I_{3}(\cdot, t)\|_{L^{p}}&\le C\sum_{j=0}^{l+1}(1+t)^{-\frac{1}{2}(l+1-j)}\int_{0}^{t}\|\p_{x}^{j}J[(b^{2}-\p_{x}^{2})^{-1}\p_{x}^{2}\psi](\cdot, t, \tau)\|_{L^{p}}d\tau\\
&\le C\sum_{j=0}^{l+1}(1+t)^{-\frac{1}{2}(l+1-j)}\left(\int_{0}^{t/2}+\int_{t/2}^{t}\right)\|\p_{x}^{j}J[(b^{2}-\p_{x}^{2})^{-1}\p_{x}^{2}\psi](\cdot, t, \tau)\|_{L^{p}}d\tau\\
&=:C\sum_{j=0}^{l+1}(1+t)^{-\frac{1}{2}(l+1-j)}(I_{3.1}+I_{3.2}).
\end{split}
\end{align}
By making the integration by parts, it follows that 
\begin{align}
\begin{split}\label{5-19}
J[(b^{2}-\p_{x}^{2})^{-1}\p_{x}^{2}\psi](x, t, \tau)&=\int_{\R}G_{0}(x-y, t-\tau)(\eta(y, \tau))^{-1}(b^{2}-\p_{y}^{2})^{-1}\p_{y}^{2}\psi(y, \tau)dy\\
&=-\int_{\R}\p_{y}\left(G_{0}(x-y, t-\tau)(\eta(y, \tau))^{-1}\right)(b^{2}-\p_{y}^{2})^{-1}\p_{y}\psi(y, \tau)dy\\
&=\int_{\R}\p_{x}G_{0}(x-y, t-\tau)(\eta(y, \tau))^{-1}(b^{2}-\p_{y}^{2})^{-1}\p_{y}\psi(y, \tau)dy\\
&\ \ \ \ -\int_{\R}G_{0}(x-y, t-\tau)\p_{y}(\eta(y, \tau)^{-1})(b^{2}-\p_{y}^{2})^{-1}\p_{y}\psi(y, \tau)dy.
\end{split}
\end{align}
Therefore, from \eqref{5-19}, Young's inequality, \eqref{eta-est2}, \eqref{heat-decay1}, Lemma~\ref{lem.est-dispersion}, Schwarz inequality, Theorem~\ref{thm.main1} and Lemma~\ref{lem.eta-decay}, we obtain
\begin{align}
\begin{split}\label{5-20}
I_{3.1}&\le \int_{0}^{t/2}\|\p_{x}^{j+1}G_{0}(\cdot, t-\tau)\|_{L^{q}}\|(b^{2}-\p_{x}^{2})^{-1}\p_{x}\psi(\cdot, \tau)\|_{L^{2}} \ \ \left(\frac{1}{p}+1=\frac{1}{q}+\frac{1}{2}\right) \\
&\ \ \ +\int_{0}^{t/2}\|\p_{x}^{j}G_{0}(\cdot, t-\tau)\|_{L^{p}}\|\p_{x}(\eta(\cdot, \tau)^{-1})(b^{2}-\p_{x}^{2})^{-1}\p_{x}\psi(\cdot, \tau)\|_{L^{1}} \\
&\le \int_{0}^{t/2}\|\p_{x}^{j+1}G_{0}(\cdot, t-\tau)\|_{L^{q}}\|\p_{x}\psi(\cdot, \tau)\|_{L^{2}} \\
&\ \ \ +\int_{0}^{t/2}\|\p_{x}^{j}G_{0}(\cdot, t-\tau)\|_{L^{p}}\|\p_{x}(\eta(\cdot, \tau)^{-1})\|_{L^{2}}\|\p_{x}\psi(\cdot, \tau)\|_{L^{2}} \\
&\le CE_{s, 1}\int_{0}^{t/2}(t-\tau)^{-\frac{1}{2}+\frac{1}{2}\left(\frac{1}{p}+\frac{1}{2}\right)-\frac{j+1}{2}}(1+\tau)^{-\frac{5}{4}+\e}d\tau \\
&\ \ \ +CE_{s, 1}\int_{0}^{t/2}(t-\tau)^{-\frac{1}{2}+\frac{1}{2p}-\frac{j}{2}}(1+\tau)^{-\frac{1}{4}}(1+\tau)^{-\frac{5}{4}+\e}d\tau\\
&\le CE_{s, 1}(1+t)^{-\frac{1}{2}+\frac{1}{2p}-\frac{j}{2}}, \ \ t\ge1.
\end{split}
\end{align}
For $i=0, 1$, by using \eqref{eta-est2}, Lemma~\ref{lem.eta-decay}, Lemma~\ref{lem.est-dispersion}, Theorem~\ref{GE} and Lemma~\ref{lem.chi-decay}, we get
\begin{align}
\begin{split}\label{5-21}
&\|\p_{x}^{j}\left(\p_{x}^{i}(\eta(\cdot, \tau)^{-1})(b^{2}-\p_{x}^{2})^{-1}\p_{x}\psi(\cdot, \tau)\right)\|_{L^{p}} \\
&\le C\sum_{n=0}^{j}\|\p_{x}^{j+i-n}(\eta(\cdot, \tau)^{-1})\|_{L^{\infty}}\|\p_{x}^{n+1}\psi(\cdot, \tau)\|_{L^{p}}\\
&\le C\sum_{n=0}^{j}(1+\tau)^{-\frac{1}{2}(j+i-n)}\left(\|\p_{x}^{n+1}u(\cdot, \tau)\|_{L^{p}}+\|\p_{x}^{n+1}\chi(\cdot, \tau)\|_{L^{p}}\right)\\
&\le CE_{s, 0}\sum_{n=0}^{j}(1+\tau)^{-\frac{1}{2}(j+i-n)}(1+\tau)^{-1+\frac{1}{2p}-\frac{n}{2}} \\
&\le CE_{s, 0}(1+\tau)^{-1+\frac{1}{2p}-\frac{1}{2}(j+i)}. 
\end{split}
\end{align}
Thus, from \eqref{5-19}, Young's inequality, \eqref{heat-decay1} and \eqref{5-21}, we have
\begin{align}
\begin{split}\label{5-22}
I_{3.2}&\le \int_{t/2}^{t}\|\p_{x}G_{0}(\cdot, t-\tau)\|_{L^{1}}\|\p_{x}^{j}\left((\eta(\cdot, \tau))^{-1}(b^{2}-\p_{x}^{2})^{-1}\p_{x}\psi(\cdot, \tau)\right)\|_{L^{p}}d\tau \\
&\ \ \ +\int_{t/2}^{t}\|G_{0}(\cdot, t-\tau)\|_{L^{1}}\|\p_{x}^{j}\left(\p_{x}(\eta(\cdot, \tau)^{-1})(b^{2}-\p_{x}^{2})^{-1}\p_{x}\psi(\cdot, \tau)\right)\|_{L^{p}}d\tau \\
&\le CE_{s, 0}\int_{t/2}^{t}(t-\tau)^{-\frac{1}{2}}(1+\tau)^{-1+\frac{1}{2p}-\frac{j}{2}}d\tau+CE_{s, 0}\int_{t/2}^{t}(1+\tau)^{-\frac{3}{2}+\frac{1}{2p}-\frac{j}{2}}d\tau\\
&\le CE_{s, 0}(1+t)^{-\frac{1}{2}+\frac{1}{2p}-\frac{j}{2}}, \ \ t\ge0.
\end{split}
\end{align}
Summing up \eqref{5-18}, \eqref{5-20} and \eqref{5-22}, we get 
\begin{equation}\label{5-23}
\|\p_{x}^{l}I_{3}(\cdot, t)\|_{L^{p}}\le CE_{s, 1}(1+t)^{-1+\frac{1}{2p}-\frac{l}{2}}, \ \ t\ge1, \ \ 0\le l\le s-2.
\end{equation}

Finally, we shall evaluate $I_{4}$. In the same way to get \eqref{5-21}, we obtain from Lemma~\ref{lem.chi-decay}
\begin{equation*}
\|\p_{x}^{j}\left((\eta(\cdot, \tau))^{-1}(b^{2}-\p_{x}^{2})^{-1}\p_{x}^{4}\chi(\cdot, \tau)\right)\|_{L^{p}}\le C|M|(1+\tau)^{-\frac{5}{2}+\frac{1}{2p}-\frac{j}{2}}. 
\end{equation*}
Therefore, by using the same argument given in the above paragraph, we have the following estimate: 
\begin{align}
\begin{split}\label{5-24}
\|\p_{x}^{l}I_{4}(\cdot, t)\|_{L^{p}}
&\le C\sum_{j=0}^{l+1}(1+t)^{-\frac{1}{2}(l+1-j)}\int_{0}^{t}\|\p_{x}^{j}J[(b^{2}-\p_{x}^{2})^{-1}\p_{x}^{4}\chi](\cdot, t, \tau)\|_{L^{p}}d\tau\\
&\le C\sum_{j=0}^{l+1}(1+t)^{-\frac{1}{2}(l+1-j)}\left(\int_{0}^{t/2}+\int_{t/2}^{t}\right)\|\p_{x}^{j}J[(b^{2}-\p_{x}^{2})^{-1}\p_{x}^{4}\chi](\cdot, t, \tau)\|_{L^{p}}d\tau\\
&\le C\sum_{j=0}^{l+1}(1+t)^{-\frac{1}{2}(l+1-j)}\biggl(\int_{0}^{t/2}\|\p_{x}^{j}G_{0}(\cdot, t-\tau)\|_{L^{p}}\|(b^{2}-\p_{x}^{2})^{-1}\p_{x}^{4}\chi(\cdot, \tau)\|_{L^{1}}d\tau\\
&\ \ \ \ +\int_{t/2}^{t}\|G_{0}(\cdot, t-\tau)\|_{L^{1}}\|\p_{x}^{j}\left((\eta(\cdot, \tau))^{-1}(b^{2}-\p_{x}^{2})^{-1}\p_{x}^{4}\chi(\cdot, \tau)\right)\|_{L^{p}}d\tau \biggl) \\
&\le C|M|\sum_{j=0}^{l+1}(1+t)^{-\frac{1}{2}(l+1-j)}\\
&\ \ \ \ \times \biggl(\int_{0}^{t/2}(t-\tau)^{-\frac{1}{2}+\frac{1}{2p}-\frac{j}{2}}(1+\tau)^{-2}d\tau+\int_{t/2}^{t}(1+\tau)^{-\frac{5}{2}+\frac{1}{2p}-\frac{j}{2}}d\tau\biggl) \\
&\le C|M|(1+t)^{-1+\frac{1}{2p}-\frac{l}{2}}, \ \ t\ge1.
\end{split}
\end{align}

Summing up \eqref{w.eq3}, \eqref{5-9}, \eqref{5-17}, \eqref{5-23} and \eqref{5-24}, we finally arrive at the desired estimate \eqref{asymp2}. This completes the proof. 
\end{proof}

\begin{proof}[\rm{\bf{Proof of Theorem~\ref{thm.main2}}}]
Summing up Proposition~\ref{prop.second1} and Proposition~\ref{prop.second2}, we immediately have \eqref{main2}. This completes the proof. 
\end{proof}

\section{Third Asymptotic Profile}

Finally in this section, we shall prove that the third asymptotic profile of the solution to \eqref{VFW} is given by $Q(x, t)$ defined by \eqref{third}. Namely, we give the proof of Theorem~\ref{thm.main3}. First, we reconsider perturbation problem for $w(x, t)$ defined by \eqref{w.eq1}:
\begin{align}\tag{\ref{w.eq2}}
\begin{split}
&w_{t}+\alpha w_{x}+(\beta \chi w)_{x}-\mu w_{xx}\\
&=-\alpha(b^{2}-\p_{x}^{2})^{-1}\p_{x}^{3}\psi-\frac{\alpha}{b^{2}}(b^{2}-\p_{x}^{2})^{-1}\p_{x}^{5}\chi-\biggl(\frac{\beta}{2}\psi^{2}\biggl)_{x}, \ \ x\in \R, \ t>0,\\
&w(x, 0)=u_{0}(x)-\chi(x, 0)=\psi_{0}(x), \ \ x\in \R.  
\end{split}
\end{align}
For this problem, from the definition of $\rho(x, t)$ by \eqref{rho} and $U$ by \eqref{U}, we can rewrite \eqref{w.eq2} as the following integral equation: 
\begin{align}\label{D}
\begin{split}
w(x, t)&=U[\psi_{0}](x, t, 0)-\frac{\beta}{2}\int_{0}^{t}U[\p_{x}(\psi^{2})(\tau)](x, t, \tau)d\tau \\
&\ \ \ -\alpha\int_{0}^{t}U\left[(b^{2}-\p_{x}^{2})^{-1}\p_{x}^{3}\psi(\tau)\right](x, t, \tau)d\tau-\frac{\alpha}{b^{2}}\int_{0}^{t}U\left[(b^{2}-\p_{x}^{2})^{-1}\p_{x}^{5}\chi(\tau)\right](x, t, \tau)d\tau\\
&=U[\psi_{0}](x, t, 0)+\int_{0}^{t}\int_{\R}\p_{x}(G_{0}(x-y, t-\tau)\eta(x, t))\rho(y, \tau)dyd\tau \\
&=:U[\psi_{0}](x, t, 0)+D(x, t).
\end{split}
\end{align} 

Our first step to prove Theorem~\ref{thm.main3} is to derive the following lemma: 
\begin{lem}
\label{lem.third-L}
Let $l$ be a non-negative integer and $p\in [1, \infty]$. Suppose $z_{0}\in L^{1}_{1}(\R)$. Then, we have
\begin{equation}
\label{third-L}
\|\p_{x}^{l}(U[\psi_{0}](\cdot, t, 0)-\theta_{0}\p_{x}(G_{0}(\cdot, 1+t)\eta(\cdot, t)))\|_{L^{p}}\le C\|z_{0}\|_{L^{1}_{1}}(1+t)^{-\frac{3}{2}+\frac{1}{2p}-\frac{l}{2}}, \ t\ge1,
\end{equation}
where $\displaystyle \theta_{0}:=\int_{\R}z_{0}(x)dx$ with $z_{0}(x)$ being defined by \eqref{initial}. 
\end{lem}
\begin{proof}
From the definition of $U$ given by \eqref{U} and $z_{0}(x)$ given by \eqref{initial}, we have 
\begin{align*}
\begin{split}
U[\psi_{0}](x, t, 0)=&\int_{\R}\p_{x}(G_{0}(x-y, t)\eta(x, t))\eta(y, 0)^{-1}\biggl(\int_{-\infty}^{y}(u_{0}(\xi)-\chi(\xi, 0))d\xi\biggl) dy\\
=&\int_{\R}\p_{x}(G_{0}(x-y, t)\eta(x, t))z_{0}(y)dy. 
\end{split}
\end{align*} 
Therefore, from the mean value theorem, there exist $c_{0}, c_{1}\in (0, 1)$ such that
\begin{align*}
\begin{split}
&\p_{x}^{l}(U[\psi_{0}](x, t, 0)-\theta_{0}\p_{x}(G_{0}(x, 1+t)\eta(x, t)))\\
&=\p_{x}^{l}U[\psi_{0}](x, t, 0)-\theta_{0}\p_{x}^{l+1}(G_{0}(x, 1+t)\eta(x, t))\\
&=\sum_{j=0}^{l+1}\begin{pmatrix} l+1 \\ j\end{pmatrix}\p_{x}^{l+1-j}\eta(x, t)\int_{\R}\p_{x}^{j}G_{0}(x-y, t)z_{0}(y)dy-\theta_{0}\sum_{j=0}^{l+1}\begin{pmatrix} l+1 \\ j\end{pmatrix}\p_{x}^{l+1-j}\eta(x, t)\p_{x}^{j}G_{0}(x, 1+t)\\
&=\sum_{j=0}^{l+1}\begin{pmatrix} l+1 \\ j\end{pmatrix}\p_{x}^{l+1-j}\eta(x, t)\int_{\R}\p_{x}^{j}G_{0}(x-y, t)z_{0}(y)dy-\theta_{0}\sum_{j=0}^{l+1}\begin{pmatrix} l+1 \\ j\end{pmatrix}\p_{x}^{l+1-j}\eta(x, t)\p_{x}^{j}G_{0}(x, t)\\
&\ \ \ \ +\theta_{0}\sum_{j=0}^{l+1}\begin{pmatrix} l+1 \\ j\end{pmatrix}\p_{x}^{l+1-j}\eta(x, t)\p_{x}^{j}G_{0}(x, t)-\theta_{0}\sum_{j=0}^{l+1}\begin{pmatrix} l+1 \\ j\end{pmatrix}\p_{x}^{l+1-j}\eta(x, t)\p_{x}^{j}G_{0}(x, 1+t)\\
&=\sum_{j=0}^{l+1}\begin{pmatrix} l+1 \\ j\end{pmatrix}\p_{x}^{l+1-j}\eta(x, t)\int_{\R}(\p_{x}^{j}G_{0}(x-y, t)-\p_{x}^{j}G_{0}(x, t))z_{0}(y)dy\\
&\ \ \ \ +\theta_{0}\sum_{j=0}^{l+1}\begin{pmatrix} l+1 \\ j\end{pmatrix}\p_{x}^{l+1-j}\eta(x, t)(\p_{x}^{j}G_{0}(x, t)-\p_{x}^{j}G_{0}(x, 1+t))\\
&=\sum_{j=0}^{l+1}\begin{pmatrix} l+1 \\ j\end{pmatrix}\p_{x}^{l+1-j}\eta(x, t)\int_{\R}\p_{x}^{j+1}G_{0}(x-c_{0}y, t) (-y) z_{0}(y)dy\\
&\ \ \ \ +\theta_{0}\sum_{j=0}^{l+1}\begin{pmatrix} l+1 \\ j\end{pmatrix}\p_{x}^{l+1-j}\eta(x, t)\p_{x}^{j}\p_{t}G_{0}(x, t+c_{1}). 
\end{split}
\end{align*}
Finally, applying Lemma~\ref{lem.eta-decay}, \eqref{eta-est1} and Lemma~\ref{lem.heat-decay}, we obtain 
\begin{align*}
&\|\p_{x}^{l}(U[\psi_{0}](\cdot, t, 0)-\theta_{0}\p_{x}(G_{0}(\cdot, 1+t)\eta(\cdot, t)))\|_{L^{p}} \\
&\le C\sum_{j=0}^{l+1}(1+t)^{-\frac{1}{2}(l+1-j)}\int_{\R}\|\p_{x}^{j+1}G_{0}(\cdot-c_{0}y, t)\|_{L^{p}}(1+|y|)|z_{0}(y)|dy \\
&\ \ \ \ +C\|z_{0}\|_{L^{1}}\sum_{j=0}^{l+1}(1+t)^{-\frac{1}{2}(l+1-j)}\|\p_{x}^{j}\p_{t}G_{0}(\cdot, t+c_{1})\|_{L^{p}}\\
&\le C\|z_{0}\|_{L^{1}_{1}}\sum_{j=0}^{l+1}(1+t)^{-\frac{1}{2}(l+1-j)}t^{-\frac{1}{2}+\frac{1}{2p}-\frac{j+1}{2}}\\
&\ \ \ \ +C\|z_{0}\|_{L^{1}}\sum_{j=0}^{l+1}(1+t)^{-\frac{1}{2}(l+1-j)}t^{-\frac{1}{2}+\frac{1}{2p}-\frac{j}{2}-\frac{1}{2}} \\
&\le C\|z_{0}\|_{L^{1}_{1}}(1+t)^{-\frac{3}{2}+\frac{1}{2p}-\frac{l}{2}}, \ t\ge1.
\end{align*}
This completes the proof. 
\end{proof}

Next, for the Duhamel term $D(x, t)$ of \eqref{D}, we have the following asymptotic relation: 
\begin{lem}
\label{lem.third-D}
Let $s\ge3$. Assume that $u_{0}\in L^{1}_{1}(\R)\cap H^{s}(\R)$, $z_{0}\in L^{1}_{1}(\R)$ and $E_{s, 0}$ is sufficiently small. Then, the estimate 
\begin{equation}
\label{third-D}
\lim_{t\to \infty}(1+t)^{1-\frac{1}{2p}+\frac{l}{2}}\|\p_{x}^{l}(D(\cdot, t)-\theta_{1}\p_{x}(G_{0}(\cdot, 1+t)\eta(\cdot, t)))\|_{L^{p}}=0
\end{equation}
holds for any $p\in [2, \infty]$ and integer $0\le l\le s-3$, where $D(x, t)$ is defined by \eqref{D}, while $\displaystyle \theta_{1}:=\int_{0}^{\infty}\int_{\R}\rho(x, t)dxdt$ with $\rho(x, t)$ defined by \eqref{rho}. 
\end{lem}
\begin{proof}
First, we shall check the well-definedness of $\theta_{1}$. Recalling the definition of $\rho(x, t)$ and $\psi(x, t)=u(x, t)-\chi(x, t)$, and using the integration by parts for the second term, we have
\begin{align}\label{theta1}
\begin{split}
\theta_{1}&=-\int_{0}^{\infty}\int_{\R}\eta(x, t)^{-1}\left(\frac{\beta}{2}\psi^{2}+\frac{2B}{b}(b^{2}-\p_{x}^{2})^{-1}\p_{x}^{2}\psi+\frac{2B}{b^{3}}(b^{2}-\p_{x}^{2})^{-1}\p_{x}^{4}\chi\right)(x, t)dxdt\\
&=-\int_{0}^{\infty}\int_{\R}\eta(x, t)^{-1}\left(\frac{\beta}{2}\psi^{2}+\frac{\beta B}{b}\chi(b^{2}-\p_{x}^{2})^{-1}\p_{x}\psi+\frac{2B}{b^{3}}(b^{2}-\p_{x}^{2})^{-1}\p_{x}^{4}\chi\right)(x, t)dxdt. 
\end{split}
\end{align}
Therefore, from \eqref{eta-est2}, \eqref{psi-L2}, \eqref{chi-decay}, \eqref{psi-Linf} and Lemma~\ref{lem.est-dispersion}, we obtain 
\begin{align}\label{well-d.theta1}
\begin{split}
|\theta_{1}|&\le C\int_{0}^{\infty}(\|\psi(\cdot, t)\|_{L^{2}}^{2}+\|\chi(\cdot, t)\|_{L^{1}}\|\p_{x}\psi(\cdot, t)\|_{L^{\infty}}+\|\p_{x}^{4}\chi(\cdot, t)\|_{L^{1}})dt \\
&\le C\int_{0}^{\infty}\left((1+t)^{-\frac{3}{2}+2\e}+(1+t)^{-\frac{3}{2}+\e}+(1+t)^{-2}\right)dt\le C.
\end{split}
\end{align}
Thus, $\theta_{1}$ is well-defined. 

Now, let us prove that the leading term of $D(x, t)$ is given by $\theta_{1}\p_{x}(G_{0}(x, 1+t)\eta(x, t))$. In the following, let $0\le l \le s-3$. From the definition of $D(x, t)$ given by \eqref{D}, it follows that 
\begin{align}\label{D-U}
\begin{split}
&\p_{x}^{l}(D(x, t)-\theta_{1}\p_{x}(G_{0}(x, 1+t)\eta(x, t))) \\
&=\int_{0}^{t}\int_{\R}\p_{x}^{l+1}(G_{0}(x-y, t-\tau)\eta(x, t))\rho(y, \tau)dyd\tau \\
&\ \ \ \ -\left(\int_{0}^{\infty}\int_{\R}\rho(y, \tau)dyd\tau\right)\p_{x}^{l+1}(G_{0}(x, 1+t)\eta(x, t))\\
&=\sum_{j=0}^{l+1}\begin{pmatrix} l+1 \\ j\end{pmatrix}\p_{x}^{l+1-j}\eta(x, t)\int_{0}^{t}\int_{\R}\p_{x}^{j}G_{0}(x-y, t-\tau)\rho(y, \tau)dyd\tau\\
&\ \ \ \ -\sum_{j=0}^{l+1}\begin{pmatrix} l+1 \\ j\end{pmatrix}\p_{x}^{l+1-j}\eta(x, t)\p_{x}^{j}G_{0}(x, 1+t)\left(\int_{0}^{\infty}\int_{\R}\rho(y, \tau)dyd\tau\right) \\
&=\sum_{j=0}^{l+1}\begin{pmatrix} l+1 \\ j\end{pmatrix}\p_{x}^{l+1-j}\eta(x, t)\int_{0}^{t}\int_{\R}\left(\p_{x}^{j}G_{0}(x-y, t-\tau)-\p_{x}^{j}G_{0}(x, 1+t)\right)\rho(y, \tau)dyd\tau\\
&\ \ \ \ -\sum_{j=0}^{l+1}\begin{pmatrix} l+1 \\ j\end{pmatrix}\p_{x}^{l+1-j}\eta(x, t)\p_{x}^{j}G_{0}(x, 1+t)\left(\int_{t}^{\infty}\int_{\R}\rho(y, \tau)dyd\tau\right)\\
&=:K_{1}(x, t)+K_{2}(x, t). 
\end{split}
\end{align}

First, we shall evaluate $K_{2}(x, t)$. By using \eqref{eta-est1}, Lemma~\ref{lem.eta-decay} and Lemma~\ref{lem.heat-decay}, we get 
\begin{align}\label{K2}
\begin{split}
\|K_{2}(\cdot, t)\|_{L^{p}}&\le C\sum_{j=0}^{l+1}(1+t)^{-\frac{l+1-j}{2}}(1+t)^{-\frac{1}{2}+\frac{1}{2p}-\frac{j}{2}}\left|\left(\int_{t}^{\infty}\int_{\R}\rho(y, \tau)dyd\tau\right)\right| \\
&\le C(1+t)^{-1+\frac{1}{2p}-\frac{l}{2}}\left|\left(\int_{t}^{\infty}\int_{\R}\rho(y, \tau)dyd\tau\right)\right|=:C(1+t)^{-1+\frac{1}{2p}-\frac{l}{2}}|R(t)|.
\end{split}
\end{align}
By modifying the discussion of \eqref{theta1} and \eqref{well-d.theta1}, and applying \eqref{eta-est2}, \eqref{main2}, \eqref{second1}, \eqref{chi-decay} and Lemma~\ref{lem.est-dispersion}, we can evaluate $R(t)$. Indeed, using the integration by parts, it follows that 
\begin{align}\label{R(t)}
\begin{split}
|R(t)|&\le C\int_{t}^{\infty}(\|\psi(\cdot, t)\|_{L^{2}}^{2}+\|\chi(\cdot, t)\|_{L^{1}}\|\p_{x}\psi(\cdot, t)\|_{L^{\infty}}+\|\p_{x}^{4}\chi(\cdot, t)\|_{L^{1}})dt \\
&\le C\int_{t}^{\infty}(1+\tau)^{-\frac{3}{2}}\log(1+\tau)^{2}d\tau=2C\int_{t}^{\infty}(1+\tau)^{-\frac{3}{2}}\log(1+\tau)d\tau\\
&=2C\left[-2(1+\tau)^{-\frac{1}{2}}\log(1+\tau)\right]_{t}^{\infty}-2C\int_{t}^{\infty}(-2)(1+\tau)^{-\frac{1}{2}}(1+\tau)^{-1}d\tau \\
&=4C(1+t)^{-\frac{1}{2}}\log(1+t)+4C\int_{t}^{\infty}(1+\tau)^{-\frac{3}{2}}d\tau \\
&=4C(1+t)^{-\frac{1}{2}}\log(1+t)+8C(1+t)^{-\frac{1}{2}} \\
&\le C(1+t)^{-\frac{1}{2}}\log(1+t), \ \ t\ge e-1.
\end{split}
\end{align} 
Hence, combining \eqref{K2} and \eqref{R(t)}, we can see that 
\begin{equation}\label{K2-2}
\|K_{2}(\cdot, t)\|_{L^{p}}\le C(1+t)^{-\frac{3}{2}+\frac{1}{2p}-\frac{l}{2}}\log(1+t), \ \ t\ge e-1. 
\end{equation}

Next, we deal with $K_{1}(x, t)$. In the same way as \eqref{K2}, we obtain 
\begin{equation}\label{K1}
\|K_{1}(\cdot, t)\|_{L^{p}}\le C\sum_{j=0}^{l+1}\begin{pmatrix} l+1 \\ j\end{pmatrix}(1+t)^{-\frac{l+1-j}{2}}\|X_{j}(\cdot, t)\|_{L^{p}},
\end{equation}
where we have set $X_{j}(x, t)$ as 
\begin{equation*}
X_{j}(x, t):=\int_{0}^{t}\int_{\R}\left(\p_{x}^{j}G_{0}(x-y, t-\tau)-\p_{x}^{j}G_{0}(x, 1+t)\right)\rho(y, \tau)dyd\tau.
\end{equation*}
For the latter sake, we shall rewrite $X_{j}(x, t)$. Recalling \eqref{Gauss} and \eqref{heat}, we can see that 
\begin{align*}
G_{0}(x-y, t-\tau)&=G(x-y-\alpha(t-\tau), t-\tau),  \\
G_{0}(x, 1+t)&=G(x-\alpha(1+t), 1+t). 
\end{align*}
Therefore, by using the change of variable, we have 
\begin{align*}
X_{j}(x, t)&=\int_{0}^{t}\int_{\R}\left(\p_{x}^{j}G_{0}(x-y, t-\tau)-\p_{x}^{j}G_{0}(x, 1+t)\right)\rho(y, \tau)dyd\tau \\
&=\int_{0}^{t}\int_{\R}\left(\p_{x}^{j}G(x-y-\alpha(t-\tau), t-\tau)-\p_{x}^{j}G(x-\alpha(1+t), 1+t)\right)\rho(y, \tau)dyd\tau \\
&=\int_{0}^{t}\int_{\R}\left(\p_{x}^{j}G(x-\alpha(1+t)-y, t-\tau)-\p_{x}^{j}G(x-\alpha(1+t), 1+t)\right)\\
&\ \ \ \ \times \rho(y+\alpha(1+\tau), \tau)dyd\tau. 
\end{align*}
Now, we take small $\e>0$ and then split the integral of the right hand side of $X_{j}(x, t)$ as follows: 
\begin{align}\label{X(x,t)}
\begin{split}
X_{j}(x, t)&=\int_{\e t/2}^{t}\int_{\R}\p_{x}^{j}G(x-\alpha(1+t)-y, t-\tau)\rho(y+\alpha(1+\tau), \tau)dyd\tau\\
&\ \ \ \ -\int_{\e t/2}^{t}\int_{\R}\p_{x}^{j}G(x-\alpha(1+t), 1+t)\rho(y+\alpha(1+\tau), \tau)dyd\tau \\
&\ \ \ \ +\int_{0}^{\e t/2}\int_{|y|\ge \e \sqrt{t}}\p_{x}^{j}G(x-\alpha(1+t)-y, t-\tau)\rho(y+\alpha(1+\tau), \tau)dyd\tau \\
&\ \ \ \ -\int_{0}^{\e t/2}\int_{|y|\ge \e \sqrt{t}}\p_{x}^{j}G(x-\alpha(1+t), 1+t)\rho(y+\alpha(1+\tau), \tau)dyd\tau \\
&\ \ \ \ +\int_{0}^{\e t/2}\int_{|y|\le \e \sqrt{t}}\left(\p_{x}^{j}G(x-\alpha(1+t)-y, t-\tau)-\p_{x}^{j}G(x-\alpha(1+t), 1+t)\right)\\
&\ \ \ \ \ \ \ \ \times \rho(y+\alpha(1+\tau), \tau)dyd\tau \\
&=:Y_{1}+Y_{2}+Y_{3}+Y_{4}+Y_{5}. 
\end{split}
\end{align}
To evaluate $X_{j}(x, t)$, for the heat kernel $G(x, t)$, we recall the following well known estimate:
\begin{equation}\label{est.Gauss}
\|\p_{t}^{k}\p_{x}^{l}G(\cdot, t)\|_{L^{p}}\le Ct^{-\frac{1}{2}+\frac{1}{2p}-\frac{l}{2}-k}, \ \ t>0.
\end{equation}
Here, we note that the time derivative of $G(x, t)$ decays faster than that of $G_{0}(x, t)$. This fact plays an important role in the estimate for $Y_{5}$ below. Also, by developing the argument given in the previous paragraph, using \eqref{eta-est2}, \eqref{main2}, \eqref{second1}, \eqref{chi-decay} and Lemma~\ref{lem.est-dispersion}, we are able to see that 
\begin{equation}\label{est.rho}
\|\p_{x}^{j}\rho(\cdot, t)\|_{L^{p}}\le CE_{s, 1}(1+t)^{-2+\frac{1}{2p}-\frac{j}{2}}\log(1+t), \ \ t\ge e-1.
\end{equation}
Now, let us evaluate $Y_{1}$ to $Y_{5}$. First for $Y_{1}$, from Young's inequality, \eqref{est.Gauss} and \eqref{est.rho}, we obtain 
\begin{align}\label{Y1}
\begin{split}
\|Y_{1}(\cdot, t)\|_{L^{p}}
&= \left\|\int_{\e t/2}^{t}\int_{\R}\p_{x}^{j}G(\cdot-\alpha(1+t)-y, t-\tau)\rho(y+\alpha(1+\tau), \tau)dyd\tau\right\|_{L^{p}}\\
&= \left\|\int_{\e t/2}^{t}\int_{\R}\p_{x}^{j}G(\cdot-y, t-\tau)\rho(y+\alpha(1+\tau), \tau)dyd\tau\right\|_{L^{p}}\\
&\le C\int_{\e t/2}^{t}\|G(\cdot, t-\tau)\|_{L^{1}}\|\p_{x}^{j}\rho(\cdot+\alpha(1+\tau), \tau)\|_{L^{p}}d\tau \\
&\le C\int_{\e t/2}^{t}(1+\tau)^{-2+\frac{1}{2p}-\frac{j}{2}}\log(1+\tau)d\tau \\
&\le C\e^{-2+\frac{1}{2p}-\frac{j}{2}}t^{-1+\frac{1}{2p}-\frac{j}{2}}\log(1+t), \ \ t\ge e-1.
\end{split}
\end{align}
Next, for $Y_{2}$, we easily have 
\begin{align}\label{Y2}
\begin{split}
\|Y_{2}(\cdot, t)\|_{L^{p}}&\le C\|\p_{x}^{j}G(\cdot-\alpha(1+t), 1+t)\|_{L^{p}}\int_{\e t/2}^{t}\int_{\R}|\rho(y+\alpha(1+\tau), \tau)|dyd\tau \\
&\le C(1+t)^{-\frac{1}{2}+\frac{1}{2p}-\frac{j}{2}}\int_{\e t/2}^{t}(1+\tau)^{-\frac{3}{2}}\log(1+\tau)d\tau \\
&\le C\e^{-\frac{3}{2}}t^{-1+\frac{1}{2p}-\frac{j}{2}}\log(1+t), \ \ t\ge e-1.
\end{split}
\end{align}
Analogously, we obtain for $Y_{3}$ that 
\begin{align}\label{Y3}
\begin{split}
\|Y_{3}(\cdot, t)\|_{L^{p}}&\le \int_{0}^{\e t/2}\int_{|y|\ge \e\sqrt{t}}\|\p_{x}^{j}G(\cdot-\alpha(1+t)-y, t-\tau)\|_{L^{p}}|\rho(y+\alpha(1+\tau), \tau)|dyd\tau \\
&\le Ct^{-\frac{1}{2}+\frac{1}{2p}-\frac{j}{2}}Z(t),
\end{split}
\end{align}
where we have defined 
\[Z(t):=\int_{0}^{\e t/2}\int_{|y|\ge \e\sqrt{t}}|\rho(y+\alpha(1+\tau), \tau)|dyd\tau.\]
In the same way as $Y_{3}$, we have the following estimate for $Y_{4}$: 
\begin{align}\label{Y4}
\begin{split}
\|Y_{4}(\cdot, t)\|_{L^{p}}\le Ct^{-\frac{1}{2}+\frac{1}{2p}-\frac{j}{2}}Z(t).
\end{split}
\end{align}
In addition, applying the Lebesgue's dominated convergence theorem, we can see that 
\begin{equation}\label{Z}
\lim_{t\to \infty}Z(t)=0
\end{equation}
because $\displaystyle \int_{0}^{\infty}\int_{\R}|\rho(x, t)|dxdt<\infty$. Finally, we shall treat $Y_{5}$. If $|y|\le \e \sqrt{t}$ and $0\le \tau \le \e t/2$, by using the mean value theorem and \eqref{est.Gauss}, we have 
\begin{align*}
\begin{split}
&\|\p_{x}^{j}G(\cdot-\alpha(1+t)-y, t-\tau)-\p_{x}^{j}G(\cdot-\alpha(1+t), 1+t)\|_{L^{p}} \\
&=\|\p_{x}^{j}G(\cdot-y, t-\tau)-\p_{x}^{j}G(\cdot, 1+t)\|_{L^{p}} \\
&\le \|\p_{x}^{j}G(\cdot-y, t-\tau)-\p_{x}^{j}G(\cdot, t-\tau)\|_{L^{p}} +\|\p_{x}^{j}G(\cdot, t-\tau)-\p_{x}^{j}G(\cdot, 1+t)\|_{L^{p}} \\
&\le C(t-\tau)^{-1+\frac{1}{2p}-\frac{j}{2}}|y|+C(t-\tau)^{-\frac{3}{2}+\frac{1}{2p}-\frac{j}{2}}(1+\tau) \\
&\le C\e t^{-\frac{1}{2}+\frac{1}{2p}-\frac{j}{2}}+Ct^{-\frac{3}{2}+\frac{1}{2p}-\frac{j}{2}}. 
\end{split}
\end{align*}
Thus, combining $\displaystyle \int_{0}^{\infty}\int_{\R}|\rho(x, t)|dxdt<\infty$ and the above estimate, we obtain  
\begin{align}\label{Y5}
\begin{split}
\|Y_{5}(\cdot, t)\|_{L^{p}}
&\le C\int_{0}^{\e t/2}\int_{|y|\le \e\sqrt{t}}\|\p_{x}^{j}G(\cdot-\alpha(1+t)-y, t-\tau)-\p_{x}^{j}G_{0}(\cdot-\alpha(1+t), 1+t)\|_{L^{p}}\\
&\ \ \ \ \ \ \ \times|\rho(y+\alpha(1+\tau), \tau)|dyd\tau \\
&\le C\e t^{-\frac{1}{2}+\frac{1}{2p}-\frac{j}{2}}+Ct^{-\frac{3}{2}+\frac{1}{2p}-\frac{j}{2}}, \ \ t>0.
\end{split}
\end{align}
Eventually, summing up \eqref{D-U}, \eqref{K2-2} trough \eqref{X(x,t)} and \eqref{Y1} through \eqref{Y5}, we arrive at 
\[\limsup_{t\to \infty}t^{1-\frac{1}{2p}+\frac{l}{2}}\|\p_{x}^{l}(D(\cdot, t)-\theta_{1}\p_{x}(G_{0}(\cdot, 1+t)\eta(\cdot, t)))\|_{L^{p}}\le C\e.\]
Therefore, we finally obtain 
\[\lim_{t\to \infty}t^{1-\frac{1}{2p}+\frac{l}{2}}\|\p_{x}^{l}(D(\cdot, t)-\theta_{1}\p_{x}(G_{0}(\cdot, 1+t)\eta(\cdot, t)))\|_{L^{p}}=0\]
because $\e>0$ can be chosen arbitrarily small. By modifying the limit in the above formula, we are able to get the desired result \eqref{third-D}. This completes the proof. 
\end{proof}

By virtue of the above Lemma's, we can derive the leading term of $w(x, t)$ as follows:  
\begin{prop}
\label{prop.third1}
Let $s\ge3$. Assume that $u_{0}\in L^{1}_{1}(\R)\cap H^{s}(\R)$, $z_{0}\in L^{1}_{1}(\R)$ and $E_{s, 0}$ is sufficiently small. Then, the estimate 
\begin{equation}
\label{lead.w}
\lim_{t\to \infty}(1+t)^{1-\frac{1}{2p}+\frac{l}{2}}\|\p_{x}^{l}(w(\cdot, t)-W(\cdot, t))\|_{L^{p}}=0
\end{equation}
holds for any $p\in [2, \infty]$ and integer $0\le l\le s-3$, where $w(x, t)$ is defined by \eqref{w.eq1} and is the solution to \eqref{w.eq2}, while $W(x, t)$ is defined by \eqref{third1}. 
\end{prop}
\begin{proof}
From the definition of $W(x, t)$, we can easily obtain 
\begin{equation}\label{ww}
W(x, t)=\theta V_{*}\biggl(\frac{x-\alpha(1+t)}{\sqrt{1+t}}\biggl)(1+t)^{-1}=\theta \p_{x}(G_{0}(x, 1+t)\eta(x, t)). 
\end{equation} 
Also, from \eqref{initial}, we get $\theta=\theta_{0}+\theta_{1}$. Therefore, by using \eqref{D} and \eqref{ww}, we have 
\begin{align*}
w(x, t)-W(x, t)&=U[\psi_{0}](x, t, 0)+D(x, t)-\theta\p_{x}(G_{0}(x, 1+t)\eta(x, t))\\
&=U[\psi_{0}](x, t, 0)-\theta_{0}\p_{x}(G_{0}(x, 1+t)\eta(x, t))+D(x, t)-\theta_{1}\p_{x}(G_{0}(x, 1+t)\eta(x, t)).
\end{align*}
Thus, summing up Lemma~\ref{lem.third-L} and Lemma~\ref{lem.third-D}, we arrive at the desired result \eqref{lead.w}.
\end{proof}

Finally, to complete the proof of Theorem~\ref{thm.main3}, we shall prove that the leading term of the perturbation $v-V$ is given by $\Psi(x, t)$ defined by \eqref{third2} as follows: 
\begin{prop}
\label{prop.third2}
Let $l$ be a non-negative integer. Then, for $|M|\le1$ and $p\in [1, \infty]$, we have
\begin{equation}
\label{lead.v-V}
\|\p_{x}^{l}(v(\cdot, t)-V(\cdot, t)-\Psi(\cdot, t))\|_{L^{p}}\le C|M|(1+t)^{-\frac{3}{2}+\frac{1}{2p}-\frac{l}{2}}, \ \ t\ge1,
\end{equation}
where $v(x, t)$ is the solution to \eqref{v}, while $V(x, t)$ and $\Psi(x, t)$ are defined by \eqref{second1} and \eqref{third2}, respectively. 
\end{prop}
\begin{proof}
From the definition of $V(x, t)$ by \eqref{second1}, it follows that 
\[V(x, t)= -\kappa dV_{*}\biggl(\frac{x-\alpha(1+t)}{\sqrt{1+t}}\biggl)(1+t)^{-1}\log(1+t)=-\kappa d\p_{x}(G_{0}(x, 1+t)\eta(x, t))\log(1+t). \]
Therefore, $V(x, t)$ satisfies the following equation: 
\begin{align}
\begin{split}\label{V.eq}
V_{t}+\alpha V_{x}+(\beta \chi V)_{x}-\mu V_{xx}+\kappa d(1+t)^{-2}V_{*}\biggl(\frac{x-\alpha(1+t)}{\sqrt{1+t}}\biggl)&=0 , \ \ x\in \R, \ \ t>0, \\
V(x, 0)&=0, \ \ x\in \R.
\end{split}
\end{align}
Since $v(x, 0)=V(x, 0)=0$, combining \eqref{v} and \eqref{V.eq} and applying Lemma~\ref{lem.formula}, we have
\begin{align}\label{6.8}
\begin{split}
v(x, t)-V(x, t)&=\int_{0}^{t}\int_{\R}\p_{x}(G_{0}(x-y, t-\tau)\eta(x, t))(\eta(y, \tau))^{-1} \\
&\ \ \ \ \times \left(\frac{2B}{b^{3}}\chi_{yy}(y, \tau)-\frac{\kappa d}{\sqrt{4\pi \mu}}(1+\tau)^{-\frac{3}{2}}e^{-\frac{(y-\alpha(1+\tau))^{2}}{4\mu(1+\tau)}}\eta(y, \tau)\right)dyd\tau. \\
&=\int_{0}^{t}\int_{\R}\p_{x}(G_{0}(x-y, t-\tau)\eta(x, t))F_{*}\left(\frac{y-\alpha(1+\tau)}{\sqrt{1+\tau}}\right)(1+\tau)^{-\frac{3}{2}}dyd\tau\\
&=\p_{x}\left(\eta(x, t)\int_{0}^{t}(1+\tau)^{-\frac{3}{2}}\int_{\R}G_{0}(x-y, t-\tau)F_{*}\left(\frac{y-\alpha(1+\tau)}{\sqrt{1+\tau}}\right)dyd\tau\right) \\
&=:\p_{x}(\eta(x, t)I(x, t)), 
\end{split}
\end{align}
where we have used \eqref{FF}. In the following, we transform $I(x, t)$ in the right hand side of this equation to lead that the main part of $v(x, t)-V(x, t)$ is given by $\Psi(x, t)$.

First, we shall prove the well-definedness of $\Psi(x, t)$ in $C^{0}([0, \infty); W^{m, p})$ for all $m\in \mathbb{N}\cup \{0\}$ and $p\in [1, \infty]$.
To do that, we prove $\Psi_{*} \in W^{m, p}(\R)$. 
It is sufficient to show 
\begin{equation}\label{6.9}
\int_{0}^{1}(G(1-\tau)*F(\tau))(x)d\tau \in W^{m, p}(\R),
\end{equation}
where $G(x, t)$ is the heat kernel defined by \eqref{Gauss}. To prove \eqref{6.9}, we split the $\tau$-integral. First, from Young's inequality, \eqref{est.Gauss}, \eqref{chi-decay} and \eqref{eta-est2}, we obtain
\begin{align}\label{6.10}
\begin{split}
\left\|\p_{x}^{m}\left(\int_{1/2}^{1}(G(1-\tau)*F(\tau))(\cdot)d\tau \right)\right\|_{L^{p}}
&\le  \int_{1/2}^{1}\|G(\cdot, 1-\tau)\|_{L^{1}}\|\p_{x}^{m}F(\cdot, \tau)\|_{L^{p}}d\tau \\
&=\|\p_{x}^{m}F_{*}\|_{L^{p}}\int_{1/2}^{1}\tau^{-\frac{3}{2}-\frac{m}{2}+\frac{1}{2p}}d\tau \le C. 
\end{split}
\end{align}
Next, we shall evaluate the latter part of the integral. Since \eqref{second2} and \eqref{second3}, by using the integration by parts twice, we can easily show 
\[\int_{\R}F_{*}(y)dy=0.\] 
Therefore, by the change of variable, Young's inequality, \eqref{est.Gauss}, \eqref{chi-decay} and \eqref{eta-est2} again, we have
\begin{align}\label{6.11}
\begin{split}
&\left\|\p_{x}^{m}\left(\int_{0}^{1/2}(G(1-\tau)*F(\tau))(\cdot)d\tau \right)\right\|_{L^{p}}\\
&=\left\|\p_{x}^{m}\int_{0}^{1/2}\int_{\R}G(x-y, 1-\tau)\tau^{-\frac{3}{2}}F_{*}\left(\frac{y}{\sqrt{\tau}}\right)dyd\tau\right\|_{L^{p}_{x}}\\
&=\left\|\int_{0}^{1/2}\int_{\R}\left(\int_{0}^{1}\p_{x}^{m+1}G(x-\theta y, 1-\tau)d\theta\right)y\tau^{-\frac{3}{2}}F_{*}\left(\frac{y}{\sqrt{\tau}}\right)dyd\tau\right\|_{L^{p}_{x}}\\
&=\int_{0}^{1/2}\tau^{-1}\int_{0}^{1}\left\|\int_{\R}\p_{x}^{m+1}\left(\frac{1}{\theta}G\left(\frac{x}{\theta}-y, \frac{1-\tau}{\theta^{2}}\right)\right)\frac{y}{\sqrt{\tau}}F_{*}\left(\frac{y}{\sqrt{\tau}}\right)dy\right\|_{L^{p}_{x}}d\theta d\tau\\
&\le \int_{0}^{1/2}\tau^{-1}\left(\int_{0}^{1}\theta^{-1}\theta^{-(m+1)}\theta^{\frac{1}{p}}\theta^{1-\frac{1}{p}+(m+1)}d\theta\right)(1-\tau)^{-\frac{1}{2}+\frac{1}{2p}-\frac{m+1}{2}}\tau^{\frac{1}{2}}\|yF_{*}\|_{L^{1}}d\tau \\
&\le C\|yF_{*}\|_{L^{1}}\left(\int_{0}^{1/2}\tau^{-\frac{1}{2}}(1-\tau)^{-\frac{1}{2}+\frac{1}{2p}-\frac{m+1}{2}}d\tau\right)\le C. 
\end{split}
\end{align}
Summing up \eqref{6.10} and \eqref{6.11}, we get \eqref{6.9}. Therefore, $\Psi_{*} \in W^{m, p}(\R)$ and thus $\Psi(x, t)$ is well-defined in $C^{0}([0, \infty); W^{m, p})$ for all $m\in
\mathbb{N}\cup \{0\}$ and $p\in [1, \infty]$. 

In what follows, let us transform $I(x, t)$ defined by \eqref{6.8}. To simplify the calculation, we set $x_{0}:=x-\alpha(1+t)$ and then note that 
\[G_{0}(x-\alpha \tau, (1+t)-\tau)=G(x_{0}, (1+t)-\tau).\] 
In the same way to get \eqref{6.11}, by using the change of variable several times, we have
\begin{align*}
\begin{split}
I(x, t)&=\int_{1}^{1+t}\tau^{-\frac{3}{2}}\int_{\R}G_{0}(x-y, (1+t)-\tau)F_{*}\left(\frac{y-\alpha \tau}{\sqrt{\tau}}\right)dyd\tau \\
&=\int_{1}^{1+t}\tau^{-\frac{3}{2}}\int_{\R}\tau^{\frac{1}{2}}G_{0}(x-\tau^{\frac{1}{2}}y-\alpha \tau, (1+t)-\tau)F_{*}\left(y\right)dyd\tau \\
&=\int_{1}^{1+t}\tau^{-\frac{3}{2}}\int_{\R}\tau^{\frac{1}{2}}G(x_{0}-\tau^{\frac{1}{2}}y, (1+t)-\tau)F_{*}\left(y\right)dyd\tau \\
&=\int_{1}^{1+t}\tau^{-\frac{3}{2}}\int_{\R}G\left(\tau^{-\frac{1}{2}}x_{0}-y, \frac{(1+t)-\tau}{\tau}\right)F_{*}\left(y\right)dyd\tau \\
&=(1+t)^{-\frac{1}{2}}\int_{\frac{1}{1+t}}^{1}\tau^{-\frac{3}{2}}\int_{\R}G\left((1+t)^{-\frac{1}{2}}\tau^{-\frac{1}{2}}x_{0}-y, \frac{1-\tau}{\tau}\right)F_{*}\left(y\right)dyd\tau \\
&=(1+t)^{-\frac{1}{2}}\int_{\frac{1}{1+t}}^{1}\tau^{-\frac{3}{2}}\int_{\R}\tau^{-\frac{1}{2}}G\left((1+t)^{-\frac{1}{2}}\tau^{-\frac{1}{2}}x_{0}-\tau^{-\frac{1}{2}}y, \frac{1-\tau}{\tau}\right)F_{*}\left(\frac{y}{\sqrt{\tau}}\right)dyd\tau \\
&=(1+t)^{-\frac{1}{2}}\int_{\frac{1}{1+t}}^{1}\tau^{-\frac{3}{2}}\int_{\R}G((1+t)^{-\frac{1}{2}}x_{0}-y, 1-\tau)F_{*}\left(\frac{y}{\sqrt{\tau}}\right)dyd\tau \\
&=(1+t)^{-\frac{1}{2}}\int_{\frac{1}{1+t}}^{1}(G(1-\tau)*F(\tau))((1+t)^{-\frac{1}{2}}x_{0})d\tau.
\end{split}
\end{align*}
Therefore, from \eqref{6.8} and $x_{0}=x-\alpha(1+t)$, it follows that 
\begin{align}\label{6.12}
\begin{split}
v(x, t)-V(x, t)=(1+t)^{-\frac{1}{2}}\p_{x}\left(\eta(x, t)\int_{\frac{1}{1+t}}^{1}(G(1-\tau)*F(\tau))\left(\frac{x-\alpha(1+t)}{\sqrt{1+t}}\right)d\tau\right).
\end{split}
\end{align}
On the other hand, from the definition of $\Psi(x, t)$, we can rewrite it as follows:
\begin{align}\label{6.13}
\begin{split}
\Psi(x, t)&=\Psi_{*}\left(\frac{x-\alpha(1+t)}{\sqrt{1+t}}\right)(1+t)^{-1} \\
&=(1+t)^{-\frac{1}{2}}\p_{x}\left(\eta(x, t)\int_{0}^{1}(G(1-\tau)*F(\tau))\left(\frac{x-\alpha(1+t)}{\sqrt{1+t}}\right)d\tau\right).
\end{split}
\end{align}
Therefore, combining \eqref{6.12} and \eqref{6.13}, we obtain
\begin{align}\label{6.14}
\begin{split}
&\|\p_{x}^{l}(v(\cdot, t)-V(\cdot, t)-\Psi(\cdot, t))\|_{L^{p}}\\
&=(1+t)^{-\frac{1}{2}}\left\|\p_{x}^{l+1}\left(\eta(\cdot, t)\int_{0}^{\frac{1}{1+t}}(G(1-\tau)*F(\tau))\left(\frac{\cdot-\alpha(1+t)}{\sqrt{1+t}}\right)d\tau\right)\right\|_{L^{p}} \\
&=(1+t)^{-1+\frac{1}{2p}-\frac{l}{2}}\left\|\p_{x}^{l+1}\left(\eta_{*}(\cdot)\int_{0}^{\frac{1}{1+t}}(G(1-\tau)*F(\tau))\left(\cdot\right)d\tau\right)\right\|_{L^{p}}. 
\end{split}
\end{align}
Finally, for all $m\in \mathbb{N}\cup \{0\}$, in the same way to get \eqref{6.11}, we can easily show
\begin{align}\label{6.15}
\begin{split}
\left\|\p_{x}^{m}\left(\int_{0}^{\frac{1}{1+t}}(G(1-\tau)*F(\tau))(\cdot)d\tau \right)\right\|_{L^{p}}
&\le C\|yF_{*}\|_{L^{1}}\left(\int_{0}^{\frac{1}{1+t}}\tau^{-\frac{1}{2}}(1-\tau)^{-\frac{1}{2}+\frac{1}{2p}-\frac{m+1}{2}}d\tau\right) \\
&\le C\int_{0}^{\frac{1}{1+t}}\tau^{-\frac{1}{2}}d\tau \le C(1+t)^{-\frac{1}{2}}. 
\end{split}
\end{align}
Eventually, we arrive at the desired estimate \eqref{lead.v-V} from \eqref{6.14} and \eqref{6.15}. 
\end{proof}

\medskip
\begin{proof}[\rm{\bf{End of the Proof of Theorem~\ref{thm.main3}}}]
Since $Q(x, t)=W(x, t)+\Psi(x, t)$, combining Proposition~\ref{prop.third1} and Proposition~\ref{prop.third2}, we immediately obtain \eqref{main3}. This completes the proof. 
\end{proof}

\section*{Acknowledgments}

The authors would like to express their sincere gratitude to Professor Hideo Kubo for his feedback and valuable advice. They also thank Professor Hideo Takaoka for his persistent support. 

\medskip
This study is partially supported by Grant-in-Aid for JSPS Research Fellow No.18J12340. 



\medskip
\par\noindent
\begin{flushleft}Ikki Fukuda\\
Division of Mathematics and Physics, \\
Faculty of Engineering, \\
Shinshu University\\
4-17-1 Wakasato, Nagano 380-8553, JAPAN\\
E-mail: i\_fukuda@shinshu-u.ac.jp
\vskip10pt
Kenta Itasaka\\
E-mail: kenta.itasaka@gmail.com
\end{flushleft}

\end{document}